%% file: neww.tex
\documentclass{amsart}

\usepackage{amsmath, amsthm, amssymb, amsfonts, mathtools}
\usepackage{graphicx,epsfig,color}
\usepackage{scrpage2}
\usepackage{exscale}

\theoremstyle{plain}

\newtheorem{theorem}{Theorem}[section]

\newtheorem{lemma}[theorem]{Lemma}

\newtheorem{corollary}[theorem]{Corollary}

\newtheorem{proposition}[theorem]{Proposition}

\theoremstyle{definition}

\newtheorem{definition}[theorem]{Definition}

\theoremstyle{remark}
\newtheorem{remark}[theorem]{Remark}

\newtheorem{example}[theorem]{Example}

\newtheorem*{proofoftheorem}{Proof of Theorem}
\newtheorem*{ak}{Acknowledgements}

\DeclareSymbolFont{AMSb}{U}{msb}{m}{n}
\DeclareMathSymbol{\N}{\mathalpha}{AMSb}{"4E}
\DeclareMathSymbol{\R}{\mathalpha}{AMSb}{"52}
\DeclareMathSymbol{\Z}{\mathalpha}{AMSb}{"5A}
\DeclareMathSymbol{\D}{\mathalpha}{AMSb}{"44}
\DeclareMathSymbol{\s}{\mathalpha}{AMSb}{"53}

\newcommand{\sX}{\scriptscriptstyle{X}}
\newcommand{\sM}{\scriptscriptstyle{M}}
\newcommand{\sN}{\scriptscriptstyle{N}}
\newcommand{\sk}{\scriptscriptstyle{\VK}}
\newcommand{\VK}{k}
\newcommand{\sZ}{\scriptscriptstyle{Z}}
\newcommand{\sY}{\scriptscriptstyle{Y}}
\newcommand{\sscr}{\scriptscriptstyle}
\newcommand{\X}{X}
\newcommand{\dx}{\de_{\sX}}
\newcommand{\sK}{\scriptscriptstyle{K}}

\DeclareMathOperator{\frs}{\mathfrak{s}}
\DeclareMathOperator{\frc}{\mathfrak{c}}
\DeclareMathOperator{\vol}{vol}

\DeclareMathOperator{\supp}{supp}

\DeclareMathOperator{\de}{d}
\DeclareMathOperator{\m}{m}
\DeclareMathOperator{\ric}{ric}
\DeclareMathOperator{\diam}{diam}

\DeclareMathOperator{\Ent}{Ent}
\DeclareMathOperator{\DyCo}{DyCpl}

\title{On the geometry of metric measure spaces with variable curvature bounds.}

\author{Christian Ketterer}
\email{christian.ketterer@math.uni-freiburg.de}

\begin{document}

\maketitle
\begin{abstract}
Motivated by a classical comparison result of J. C. F. Sturm we introduce a curvature-dimension condition $CD(k,N)$ for general metric measure spaces 
and variable lower curvature bound $\VK$. 
In the case of non-zero constant lower curvature our approach coincides  with the celebrated condition that was proposed by
K.-T. Sturm in \cite{stugeo2}.
We prove several geometric properties as sharp Bishop-Gromov volume growth comparison or a sharp generalized Bonnet-Myers theorem (Schneider's Theorem). Additionally, our curvature-dimension condition is stable with respect to measured Gromov-Hausdorff convergence, and it is stable with respect to tensorization of finitely many metric measure spaces provided a non-branching condition is assumed.
\end{abstract}
\noindent
\tableofcontents

\input{in_v4.tex}

\small{
\bibliography{new}

\bibliographystyle{amsalpha}
}
\end{document}

%% file: in_v4.tex
\section{Introduction}
Metric measure spaces with generalized lower Ricci curvature bounds have become objects of interest in various fields of mathematics.
Since Lott, Sturm and Villani introduced the so-called curvature-dimension condition $CD(K,N)$ for $K\in\mathbb{R}$ 
and $N\in [1,\infty]$ via displacement convexity of the Shanon and Reny entropy on the 
$L^2$-Wasserstein space 
\cite{lottvillani, stugeo1, stugeo2}
a rather complete picture of the geometric and analytic properties of these spaces has been developed (e.g. \cite{rajala2, giglisplittingshort, agsriemannian, agsbakryemery, erbarkuwadasturm}). Their approach is based on and inspired by
recent fundamental breakthroughs in the theory of optimal transport (e.g. \cite{brenier, mccann, CMS, otto}).

However, the condition of lower bounded Ricci curvature is also very retrictive. 
Neither non-compact smooth Riemannian manifolds do admit a global lower curvature bound in general, nor does Hamilton's Ricci flow in general. 
Moreover, one cannot exceed the information that is encoded by the constant curvature bound. Therefore the regime of results is limited.
However, in the context of smooth Riemannian manifolds variable lower Ricci curvature play an importan role. For instance, one can deduce 
refined statements for the geometry of the space, e.g. \cite{veysseire, aubry, pengfeiwang, petersenwei, petersensprouse, schneider}.
Therefore, it seems natural to ask for a suitable extension of the theory of Lott, Sturm and Villani. 
For dimension independent situations a definition is proposed by Sturm in \cite{sturmvariable}.
But to deduce finer geometric results one also must bring a dimension bound into play.

In this article we will focus on the finite dimensional case and introduce a curvature-dimension condition $CD(\VK,N)$ for metric measure spaces $(X,\de_{\sX},\m_{\sX})$ 
where the lower curvature bound $\VK:X\rightarrow\mathbb{R}$ is a lower semi-continuous function. Before we describe our approach,
let us remind that Lott, Sturm and Villani define the curvature-dimension condition $CD(0,N)$ of an arbitrary metric measure space $(X,\de_{\sX},\m_{\sX})$ via displacement convexity for the $N$-R\'eny entropy functional
\begin{align*}
S_N(\varrho\m_{\sX})=-\int_X\varrho^{1-\frac{1}{N}}d\m_{\sX}.
\end{align*}
(The definitions in \cite{lottvillani} and in \cite{stugeo2} slightly differ.)
In \cite{stugeo2} Sturm gave a definition of $CD(K,N)$ for general $K\in\mathbb{R}$ via so-called
\textit{distorted displacement convexity} (see also \cite{viltot}). 
This approach involves the concept of modified volume distortion coefficients $\tau_{\sk,\sN}^{\sscr{(t)}}(\theta)$ that 
do not come from a linear ODE but are motivated by the geometry of Riemannian manifolds. They
capture the geometric fact that Ricci curvature of a tangent vector $v$ is the mean value of sectional curvatures of planes intersecting in $v$. 
Roughly speaking, non-zero curvature only happens perpendicular to $v$.
Our idea is to introduce generalized volume distortion coefficients as follows. We define
\begin{align*}
\tau_{\sk_{\gamma},\sN}^{\sscr{(t)}}(\theta)=t^{\frac{1}{N}}\left[\sigma_{\sk_{\gamma},\sN-1}^{\sscr{(t)}}(\theta)\right]^{\frac{N-1}{N}}
\end{align*}
where $\VK_{\gamma}(t\theta)=\VK\circ\gamma(t)$, $\gamma:[0,1]\rightarrow X$ is a constant speed geodesic and $\sigma_{\sk_{\gamma},\sN}^{\sscr{(t)}}(\theta)$ is the solution of 
\begin{eqnarray}\label{jaffa}
&u''(t)+\frac{\VK(\gamma(t))}{N}\theta^2u=0&
\end{eqnarray}
with $u(0)=0\mbox{ and }u(1)=1$ where $\theta=|\dot{\gamma}|$.
We remark, that in the case of constant curvature $\VK=K$ this yields 
\begin{align*}
\sigma_{\sK,\sN}^{\sscr{(t)}}(|\dot{\gamma}|)=\frac{\sin_{\sK/\sN}(t|\dot{\gamma}|)}{\sin_{\sK/\sN}(|\dot{\gamma}|)}
\end{align*}
that is precisely the definition of Sturm in \cite{stugeo2}. 

A key property of the distortion coefficients is their monotonicity w.r.t. $\VK$ which is a particular consequence of a classical comparison result of J. C. F. Sturm for 1-dimensional Sturm-Liouville type operators. 
\begin{theorem}[J. C. F. Sturm's comparison theorem]
Let $\kappa,\kappa':[a,b]\rightarrow \mathbb{R}$ be continuous function such that ${\kappa}'\geq {\kappa}\mbox{ on } [a,b]$ and $\frs_{\kappa'}>0$ on $(a,b]$. Then $\mathfrak{s}_{\kappa}\geq \mathfrak{s}_{{\kappa}'}$ on $[a,b]$.
\end{theorem}
\noindent
$\frs_{\kappa}$ is a solution of (\ref{jaffa}) with $\VK/N=\kappa$ and $\gamma(t)=t$, an initial condition $u(0)=0$ and $u'(0)=1$.
The theorem is well-known in the context of Riemannian manifolds and smooth Jacobi field calculus. Its geometric counterpart is the celebrated Rauch comparison theorem.

In particular, from generalized distortion coefficients we also obtain a new characterization of the differential inequality 
$
u''\leq -\VK u
$
(see Proposition \ref{central}) that appears naturally in connection with lower curvature bounds on smooth Riemannian manifolds.  

Then our curvature-dimension condition takes the following form. Let $(X,\de_{\sX},\m_{\sX})$ be a metric measure space as in Definition \ref{assmms} and assume for simplicity that
for $\m_{\sX}^2$-a.e. pair $(x,y)$ there exists a unique geodesic. Then $(X,\de_{\sX},\m_{\sX})$ satisfies the condition $CD(\VK,N)$ for
$N\geq 1$ and a lower semi-continuous function $\VK:X\rightarrow \mathbb{R}$ if for any pair of absolutely continuous probability measures $\mu_0$ and $\mu_1$ on $X$ there exists a dynamical optimal coupling
$\Pi\in\mathcal{P}(\mathcal{G}(X))$ such that
\begin{align*}
\varrho_t(\gamma_{t})^{-\frac{1}{N}}\geq \tau_{\sk^-_{\gamma},\sN'}^{\sscr{(1-t)}}(|\dot{\gamma}|)\varrho_0(\gamma_{0})^{-\frac{1}{N}}+\tau_{\sk^+_{\gamma},\sN}^{\sscr{(t)}}(|\dot{\gamma}|)\varrho_1(\gamma_{0})^{-\frac{1}{N}}.
\end{align*}
for all $t\in[0,1]$ and $\Pi$-a.e. geodesic $\gamma$. Here $\VK^+_{\gamma}=\VK_{\gamma}$ and $\VK^-_{\gamma}=\VK_{\gamma^-}$ where $\gamma^-$ is the time reverse reparametrization of $\gamma$. 
$\varrho_t$ is the density of the push-forward of $\Pi$ under the map $\gamma\mapsto \gamma_{t}$. If we replace $\tau_{\sk,\sN}$ by $\sigma_{\sk/\sN}$ 
we say $X$ satisfies the reduced curvature-dimension condition $CD^*(k,N)$.
Let us emphasize that we do not assume any non-branching assumption for the metric measure space in general, and we also do not assume a quadratic Cheeger as in \cite{agsriemannian} 
or an a priori lower curvature bound as in \cite{sturmvariable}.

This is the first part of two articles where we investigate the geometric and and analytic consequences of our curvature-dimension condition. The main results in this article are
\begin{itemize}
 \item The condition $CD(\VK,N)$ for $N\in [1,\infty)$ implies $CD(\VK,\infty)$ in the sense of \cite{sturmvariable} (Proposition \ref{htht}).
 \item For Riemannian manifolds the curvature-dimension condition $CD(\VK,N)$ is equivalent to a lower bound $\VK$ for the Ricci curvature and an upper bound $N$ for the dimension (Theorem \ref{smoothcase}). 
 \item A generalized Brunn-Minkowski theorem and a generalized Bishop-Gromov comparison theorm hold (Theorem \ref{bm}, Theorem \ref{useful}, Theorem \ref{bg}). The latter results in particular yields a local volume doubling property and finite Hausdorff dimension. 
 \item A generalized Bonnet-Myers theorem (Theorem \ref{stheorem}). This is a non-smooth version of a result by R. Schneider \cite{schneider} (see also \cite{ambrose, galloway}). It states that if the curvature doesn't decreasing too quickly for large distances from
 a point, then the space is compact. There are also similar statements in the context of smooth Finsler manifolds and for the Bakry-Emery Ricci tensor in a smooth context \cite{mihairadu, zhsh}. 
 \item The curvature-dimension condition is stable with respect to measured Gromov-Hausdorff convergence (Theorem \ref{stabilitytheorem}). In particular, it implies that any family of compact Riemannian manifolds with uniform upper bound for the dimension, 
 uniform upper bound for the diameter and equi-continuous lower Ricci curvature bounds that are uniformily bounded from below admit a converging subsequence such that the lower Ricci curvature bounds converge 
 uniformily to a continuous function that is
 a lower Ricci curvature bound for the limit space.
 \item The curvature-dimension condition is stable under tensorization of finitely many metric measure spaces provided a non-branhing assumption is satisfied (Theorem \ref{tensorization}).
 \item The reduced curvature-dimension condition admits a globalization property (Theorem \ref{globalization}).
\end{itemize}
In the forthcoming addendum to this article we also investigate variants of the condition $CD(\VK,N)$. Namely, 
following \cite{erbarkuwadasturm, ohtmea} we introduce
an entropic curvature-dimension condition and a measure contraction property as well as an $EVI_{\VK,N}$-condition for gradient flows on metric spaces where $\VK$ is a lower semi-continuous function. 
We will investigate their relation to each other and also to the reduced curvature-dimension condition presented in this paper. 
Provided stronger regularity assumptions we establish various equivalences and consequences.

Additionally, considering the recent approach of Cavalletti and Mondino in \cite{cavmon} to prove isoperimetric inequalities and various other functional inequalities in the context 
of non-branching $CD$-spaces with constant curvature bound
our appoach seems very well adapted for tranforming their ideas to a non-constant curvature setting.

In the second section of this paper we will present necessary preliminaries of optimal transport, Wasserstein calculus and geometry of metric spaces.
In section 3 we will introduce generalized distortion coefficients and we will present a new characterization of $\kappa u$-convexity of a function $u$.
In section 4 we give the definition of $CD(\VK,N)$ in the general context of metric measure spaces, and in particular we will prove that is consistent with Sturm's definition in \cite{sturmvariable}.
The topic of section 5 will be the geometric consequences of the curvature-dimension condition. 
In section 6, 7 and 8 we will prove the stability property, the tensorization property under a branching assumption, and the globalization property of the reduced curvature-dimension condition, respectively. 
\begin{ak}
The author would like to thank Yu Kitabeppu for his interest and many fuitful discussions. Major parts of this work have been written 
during the Junior Trimester Programm "Optimal Transport" at the Hausdorff Institute of Mathematics (HIM) in Bonn. The author also wish to thank HIM for 
the excellent working condition and the
stimulation and open atmosphere. 
\end{ak}

\section{preliminaries}

\begin{definition}[Metric measure space]\label{assmms}
Let $(X,\de_{\sX})$ be a complete and separable metric space, 
and let $\m_{\sX}$ be a locally finite Borel measure on $(X,\de_{\sX})$.
That is,  for all $x\in X$ there exists $r>0$ such that $\m_{\sX}(B_r(x))\in(0,\infty)$.
Let $\mathcal{O}_{\sX}$ and $\mathcal{B}_{\sX}$ be the topology of open sets and the family of Borel sets, respectively.
A triple $(X,\de_{\sX},\m_{\sX})$ will be called \emph{metric measure space}. 
We assume that $\m_{\sX}(X)\neq 0$.
\end{definition}
$(X,\de_{\sX})$ is called \textit{length space} 
if $\de_{\sX}(x,y)=\inf \mbox{L} (\gamma)$ for all $x,y\in X$, 
where the infimum runs over all rectifiable curves $\gamma$ in $X$ connecting $x$ and $y$. 
$(X,\de_{\sX})$ is called \textit{geodesic space} 
if every two points $x,y\in X$ are connected by a curve $\gamma$ such that $\de_{\sX}(x,y)=\mbox{L}(\gamma)$.
Distance minimizing curves of constant speed are called \textit{geodesics}. 
A length space, which is complete and locally compact, is a geodesic space and proper (\cite[Theorem 2.5.23 ]{bbi}).
Rectifiable curves always admit a reparametrization proportional to arc length, and therefore become Lipschitz curves.
In general, we assume that a geodesic $\gamma:[0,1]\rightarrow X$ is parametrized proportional to its length, and
the set of all such geodesics $\gamma:[0,1]\rightarrow X$ is denoted with $\mathcal{G}(X)$. 
The set of all Lipschitz curves $\gamma:[0,1]\rightarrow X$ parametrized proportional to arc-length is denoted with $\mathcal{LC}(X)$.
$(X,\de_{\sX})$ is called \textit{non-branching} 
if for every quadruple $(z,x_0,x_1,x_2)$ of points in $X$ for which $z$ is a midpoint of $x_0$ and $x_1$ as well as of $x_0$ and $x_2$, it follows that $x_1=x_2$.
\smallskip

$\mathcal{P}(X)$ denotes the space of probability measures on $(X,\mathcal{B}_{\sX})$, and
$\mathcal{P}_2(\X,\dx)=:\mathcal{P}_2(X)$ denotes the \textit{$L^2$-Wasserstein space} of probability measures $\mu$ on $(\X,\mathcal{B}_{\sX})$ with finite second moments, which means that $\int_\X\dx^2(x_0,x)d\mu(x)<\infty$
for some (hence all) $x_0\in \X$. The \textit{$L^2$-Wasserstein distance} $\de_{W}(\mu_0,\mu_1)$ between two probability measures
$\mu_0,\mu_1\in\mathcal{P}_2(\X)$ is defined as
\begin{equation}\label{wassersteindist}
\de_{W}(\mu_0,\mu_1)=\sqrt{\inf_{\pi}\int_{\X\times \X}\dx^2(x,y)\,d\pi(x,y)
}.
\end{equation}
Here the infimum ranges over all \textit{couplings} of $\mu_0$ and $\mu_1$, 
i.e. over all probability measures on $\X\times \X$ with marginals $\mu_0$ and $\mu_1$. $(\mathcal{P}_2(\X),\de_{W})$ is a 
complete separable metric space. The subspace of $\m_{\sX}$-absolutely continuous measures is denoted by $\mathcal{P}_2(\X,\m_{\sX})=:\mathcal{P}_2(\m_{\sX})$. 
%
%
%
%
%
A minimizer of (\ref{wassersteindist}) always exists and is called \textit{optimal coupling} between $\mu_0$ and $\mu_1$. 
\smallskip

A probability measure $\Pi$ on $\mathcal{G}(X)$ is called \emph{dynamical optimal transference plan} if and only if the probability measure $(e_0,e_1)_*\Pi$ on $\X\times \X$ is an optimal coupling of the probability measures
$(e_0)_*\Pi$ and $(e_1)_*\Pi$ on $\X$.
Here and in the sequel $e_t:\Gamma(\X)\to \X$ for $t\in[0,1]$ denotes the evaluation map $\gamma\mapsto \gamma_t$. An absolutely continuous curve $\mu_t$ in $\mathcal{P}_2(X,\m_{\sX})$ is a geodesic 
if and only if there is a dynamical optimal
transference plan $\Pi$ such that $(e_t)_*\Pi=\mu_t$. We write $\DyCo(\mu_0,\mu_1)$ for the set of dynamcial optimal transference plans between $\mu_0$ and $\mu_1$.
\smallskip

Let us recall the notion of \textit{Markov kernel}. Let $(Y,\de_{\sY})$ be a separable and complete metric space. A Markov kernel is a map $Q:Y\times \mathcal{B}_{\sY}\rightarrow [0,1]$ with the following properties. 
$Q(y,\cdot)$ is a probability measure for each $y\in Y$. The function $Q(\cdot,A)$ is measurable for each $A\in \mathcal{B}_{\sX}$.
\smallskip

\begin{lemma}
For each pair $\mu_0, \mu_1\in \mathcal{P}_2(X)$ there exists a dynamical optimal coupling $\Pi$ such that 
\begin{align*}
\de_{W}(\mu_0,\mu_1)^2=\int \de_{\sX}(\gamma(0),\gamma(1))d\Pi(\gamma).
\end{align*}
and there exist Markov kernels $\Pi_{x_0,x_1}$, $\Pi_{x_0}$ and $\Pi_{x_1}$ such that
\begin{align*}
d\Pi(\gamma)=d\Pi_{x_0,x_1}(\gamma)d\pi(x_0,x_1)=d\Pi_{x_0}(\gamma)d\mu_{0}(x_0)=d\Pi_{x_1}(\gamma)d\mu_{1}(x_1)
\end{align*}
where $(e_0,e_1)_{\star}\Pi=:\pi$.
\end{lemma}
\begin{proof}
For the existence of an dynamical optimal coupling, see \cite{viltot}. The existence of the corresponding Markov kernels comes from the existence of regular conditional probability measures.
\end{proof}

\section{$\kappa u$-convexity}\noindent
Let $\kappa:[a,b]\rightarrow \mathbb{R}$ be a continuous function. We study solutions to
\begin{align}\label{ode}
v''+\kappa v=0.
\end{align} 
The generalized $\sin$-functions $\frs_{\kappa}:[a,b]\rightarrow \mathbb{R}$
is the unique solution of (\ref{ode})
such that $\mathfrak{s}_{\kappa}(a)=0$ and $\mathfrak{s}_{\kappa}'(a)=1$.
The generalized $\cos$-function is $\mathfrak{c}_{\kappa}=\frs_{\kappa}'$.
Solutions of (\ref{ode}) depend continuously on the coefficient $\kappa$. More precisely, for each $\epsilon>0$ there exists $\delta>0$ such that $|\kappa-\kappa'|_{\infty}<\delta$ implies $|\frs_{\kappa}-\frs_{\kappa'}|_{\infty}<\epsilon$
where $\kappa,\kappa':[a,b]\rightarrow \mathbb{R}$ are continuous.
If $\gamma(t)=(1-t)a+tb$ and $v:[a,b]\rightarrow \mathbb{R}$ is any solution of (\ref{ode}), then $v\circ \gamma=u:[0,1]\rightarrow \mathbb{R}$ solves 
\begin{align}\label{ode2}
u''+{\kappa}\circ\gamma |\dot{\gamma}|^2u=0.
\end{align}
In particular, $\frs_{\kappa}(\gamma_t)$ solves (\ref{ode2}) with $\frs_{\kappa}(\gamma_0)=0$ and $\frac{d}{dt}|_{t=0}\frs_{\kappa}(\gamma_t)=|\dot{\gamma}(0)|=b-a$.
\smallskip\\
The next theorem is well-known.
\begin{theorem}[J. C. F. Sturm's comparison theorem]
Let $\kappa,\kappa':[a,b]\rightarrow \mathbb{R}$ be continuous function such that ${\kappa}'\geq {\kappa}\mbox{ on } [a,b]$ and $\frs_{\kappa'}>0$ on $(a,b]$. Then $\mathfrak{s}_{\kappa}\geq \mathfrak{s}_{{\kappa}'}$ on $[a,b]$.
\end{theorem}
\noindent
\begin{theorem}[Sturm-Picone oscillation theorem]
Let $\kappa,\kappa':[a,b]\rightarrow \mathbb{R}$ be continuous such that ${\kappa}'\geq {\kappa}\mbox{ on } [a,b]$. Let $u$ and $v$ be solutions of (\ref{ode}) with respect to $\kappa$ and $\kappa'$ respectively. If $u(a)=u(b)=0$
and $u>0$ on $(a,b)$, then either $u=\lambda v$ for some $\lambda>0$ or there exists $x_1\in (a,b]$ such that $v(x_1)=0$.
\end{theorem}
\begin{definition}[generalized distortion coefficients]\label{distort}
Consider $\kappa:[0,{L}]\rightarrow \mathbb{R}$ that is continuous and $\theta \in (0,L]$. 
Then
\begin{align*}
\sigma_{\kappa}^{\sscr{(t)}}(\theta)=\begin{cases}
\frac{\frs_{{\kappa}}(t\theta)}{\frs_{{\kappa}}(\theta)}& \mbox{ if }\frs_{\kappa}|_{(0,\theta]}>c>0,\\
\infty & \mbox{ otherwise }.
\end{cases}
\end{align*} 
We also define $\pi_{\kappa}=\sup\left\{t\in[0,L]:\frs_{\kappa}(s)>0 \mbox{ for all }s\leq t\right\}.$
If $\sigma_{\kappa}^{\sscr{(t)}}(\theta)<\infty$, $t\mapsto\sigma_{\kappa}^{\sscr{(t)}}(\theta)$ is a solution of 
\begin{align}\label{klebeband}
u''(t)+\kappa(t\theta)\theta^2u(t)=0
\end{align}
satisfying $u(0)=0$ and $u(1)=1$.
\end{definition}
\begin{proposition}\label{monotonicity}
$\sigma^{\sscr{(t)}}_{\kappa}(\theta)$ is non-decreasing with respect to $\kappa:[0,\theta]\rightarrow \mathbb{R}$. More precisely
\begin{align*}
\kappa(x)\geq\kappa'(x) \ \forall x\in [0,\theta]\ \ \mbox{implies} \ \  \sigma^{\sscr{(t)}}_{\kappa}(\theta)\geq \sigma^{\sscr{(t)}}_{\kappa'}(\theta)\ \forall t\in [0,1].
\end{align*}
\end{proposition}
\begin{proof}
Consider $\sigma^{\sscr{(t)}}_{\kappa}(\theta)$ and $\sigma^{\sscr{(t)}}_{{\kappa}'}(\theta)$ for $\kappa$ and ${\kappa}'$ such that $\kappa(t)\geq {\kappa}'(t)$ for all $t\in[0,1]$. 
By Sturm-Picone oscillation theorem $\sigma^{\sscr{(t)}}_{\kappa}(\theta)=\infty$ implies $\sigma^{\sscr{(t)}}_{{\kappa}'}(\theta)=\infty$. 
Hence, we only need to check the case when $\sigma^{\sscr{(t)}}_{\kappa}(\theta)<\infty$ and $\sigma^{\sscr{(t)}}_{{\kappa}'}(\theta)<\infty$.
\smallskip \\
We use the idea of the proof of Theorem 14.28 in \cite{viltot}.
We know that $\sigma_{\kappa}^{\sscr (0)}(\theta)=\sigma_{{\kappa}'}^{\sscr (0)}(\theta)=0$ and $\sigma_{\kappa}^{\sscr (1)}(\theta)=\sigma_{{\kappa}'}^{\sscr (1)}(\theta)=1$. 
Consider ${\sigma_{{\kappa}'}^{\sscr (t)}(\theta)}/{\sigma_{\kappa}^{\sscr (t)}(\theta)}=:h(t)$ for $t\in (0,1]$. We know that
$h(1)=1$ and L'Hospital's rule yields
\begin{align*} 
\lim_{t\downarrow 0}h(t)=\frac{\frs_{\kappa}(\theta)}{\frs_{\kappa'}(\theta)}\lim_{t\downarrow 0}\frac{\frc_{\kappa'}(t\theta)}{\frc_{\kappa}(t\theta)}=\frac{\frs_{\kappa}(\theta)}{\frs_{\kappa'}(\theta)}\leq 1.
\end{align*}
\smallskip
\\
Hence, it is sufficient to check that $h(t)$ has no local maximum in $(0,1)$. 
For this reason, first we assume that $\kappa>{\kappa}'$. 
Set $\sigma_{{\kappa}'}^{\sscr{(t)}}(\theta)=f$ and $\sigma_{\kappa}^{\sscr{(t)}}(\theta)=g$.
Assume there is a maximum in $t_0\in (0,1)$. Hence, $(f/g)'(t_0)=0$ and $(f/g)''(t_0)\leq 0$.
We compute the second derivative of $f/g$.
\begin{align*}
\left(\frac{f}{g}\right)''
&
=\frac{f''g^3-g''fg^2}{g^4}+\frac{2gg'fg'-2g'f'g^2}{g^4}\\
&=-{\kappa'}\theta^2\frac{f}{g}+\kappa\theta^2\frac{f}{g}-\frac{2gg'}{g^2}\frac{f'g-fg'}{g^2}
\end{align*}
and therefore
\begin{align*}
\left(\frac{f}{g}\right)''(t_0)=({\kappa'}(t_0\ \theta)-\kappa(t_0\theta))\theta^2\frac{f(t_0)}{g(t_0)}>0.
\end{align*}
The case where $\kappa\geq {\kappa}'$ follows from that if we replace $\kappa$ by $\kappa+\epsilon$. 
Then $\sigma_{\kappa+\epsilon}^{\sscr{(t)}}(\theta)$ converges uniformly to $\sigma_{\kappa}^{\sscr{(t)}}(\theta)$ if $\epsilon\rightarrow 0$.
\end{proof}
\begin{proposition}\label{continuity}
For $\theta\in(0,L]$ and $t\in(0,1)$ the map $\kappa\in (C([0,L]),|\cdot|_{\infty})\mapsto \sigma_{\kappa}^{\sscr{(t)}}(\theta)\in \mathbb{R}_{\geq 0}\cup\left\{\infty\right\}$ is continuous 
where $\mathbb{R}_{\geq 0}\cup\left\{\infty\right\}$ is equipped with the usual topology.
\end{proposition}
\begin{proof}
If all the distortion coefficients are finite,
this follows from the stability of (\ref{ode}) under uniform changes of $\kappa$. We only have to check the following. 
If $\kappa_n\rightarrow \kappa$ with respect to $|\cdot|_{\infty}$, and if $\sigma_{\kappa}^{\sscr{(t)}}(\theta)=\infty$,
then $\sigma_{\kappa_n}^{\sscr{(t)}}(\theta)\uparrow \infty$. If $\sigma_{\kappa}^{\sscr{(t)}}(\theta)=\infty$, 
then there exists $r\leq \theta$ such that $\frs_{\kappa}(r)=0$. If $r<\theta$, 
then by the stability property $\frs_{\kappa_n}(r_n)=0$ for some $r_n<\theta$ and $n\in\mathbb{N}$ sufficiently large.
Hence, $\sigma_{\kappa_n}^{\sscr{(t)}}(\theta)=\infty$ for $n$ sufficiently large. Otherwise $r=\theta$ and $\frs_{\kappa}>0$ on $(0,\theta)$. 
Again by stability it follows that $\frs_{\kappa_n}(\theta)\rightarrow 0$ and 
$\frs_{\kappa_n}\rightarrow \frs_{\kappa}$ w.r.t. $|\cdot|_{\infty}$ if $n\rightarrow \infty$.
Therefore, for any compact $J\subset (0,1)$ there exits $n_0$ such that for each $n\geq n_0$ we have $\frs_{\kappa_n}(\cdot\theta)|_J>c>0$ for some $c>0$. 
Hence, $\sigma_{\kappa_n}^{\sscr{(t)}}(\theta)\uparrow \infty$ for each $t\in(0,1)$.
\end{proof}
\begin{lemma}\label{ordinary}
Let $a,b\in\mathbb{R}_{\geq 0}$ and $\kappa:[0,\theta]\rightarrow \mathbb{R}$ as before. If $\sigma_{\kappa}^{\sscr{(t)}}(\theta)<\infty$, then 
\begin{align}\label{kkl}
v(t)=\sigma_{{\kappa}^{\sscr{-}}}^{\sscr{(1-t)}}(\theta)a+\sigma_{{\kappa}^{\sscr{+}}}^{\sscr{(t)}}(\theta)b
\end{align}
 solves
(\ref{klebeband}) in the distributional sense
satisfying $u(0)=a$ and $u(1)=b$. 
\end{lemma}
\begin{remark}
Given $\kappa$ as above we set ${\kappa}^{\sscr{-}}=\kappa\circ\phi$ where $\phi(t)=b+a-t$. We also write
$\kappa=:\kappa^{\sscr{+}}$. $\sigma_{\kappa}^{\sscr{(t)}}(\theta)<\infty$ if and only if  $\sigma_{\kappa^{\sscr{-}}}^{\sscr{(t)}}(\theta)<\infty$. This follows from Sturm's oscillation theorem.
\smallskip\\
To see this we
assume $\sigma_{\kappa}^{\sscr{(t)}}(\theta)=\infty$ and $\sigma_{\kappa^-}^{\sscr{(t)}}(\theta)$ is finite. Then $\frs_{\kappa}$ has a zero in $[0,\theta]$ and $\frs_{\kappa^-}$ has no zero in $[0,\theta]$.
But $\frs_{\kappa}(t)$ and $\frs_{\kappa}(\theta-t)$ are solutions of $u''+\kappa u=0$, and therefore Sturm's oscillation theorem yields a contradiction.
\end{remark}
\begin{proof}
We have
$$v''(t)=-{\kappa}^{\sscr{-}}((1-t)\theta)\theta^2\sigma_{{{\kappa}}^{-}}^{\sscr{(1-t)}}(\theta)a-\kappa(t\theta)\theta^2\sigma_{{\kappa}^{\sscr{+}}}^{\sscr{(t)}}(\theta)b$$
and
$${\kappa}^{-}((1-t)\theta)=\kappa^{\sscr{+}}\circ\phi((1-t)\theta)=\kappa^{\sscr{+}}(\theta-(1-t)\theta))=\kappa^{\sscr{+}}(t\theta).$$
Hence (\ref{kkl}) solves (\ref{klebeband}) in the classical sense satisfying the right boundary condition. 
\end{proof}
\begin{proposition}\label{central}
Let $\kappa:[a,b]\rightarrow \mathbb{R}$ be 
continuous and $u:[a,b]\rightarrow \mathbb{R}_{\geq 0}$ be an upper semi-continous. Then the following three statements are equivalent:
\begin{itemize}
 \item[(i)]$u''+\kappa u\leq 0$ in the distributional sense, that is 
\begin{align}\label{distributional}
\int_a^b\varphi''(t)u(t)dt\leq -\int_a^b\varphi(t)\kappa(t)u(t)dt
\end{align}
for any $\varphi\in C_0^{\infty}\left((a,b)\right)$ with $\varphi\geq 0$.
\smallskip
 \item[(ii)] It holds 
\begin{align}
u(\gamma(t))\geq (1-t)u(\gamma(0))+tu(\gamma(1)+\int_0^1g(t,s)\kappa(\gamma(s))\theta^2 u(\gamma(s)) ds
\end{align}
for any constant speed geodesic $\gamma:[0,1]\rightarrow [a,b]$ where $\theta=|\dot{\gamma}|=\mbox{L}(\gamma)$ with
$g(s,t)$ beeing the Green function of $[0,1]$. 
\smallskip
 \item[(iii)] There is a constant $0<L\leq b-a$ such that
 \begin{align}\label{kuconcavity}
u(\gamma(t))\geq \sigma^{\sscr{(1-t)}}_{\kappa^{\sscr{-}}_{\gamma}}(\theta)u(\gamma(0))+\sigma^{\sscr{(t)}}_{\kappa^{\sscr{+}}_{\gamma}}(\theta)u(\gamma(1))
 \end{align}
for any constant speed geodesic $\gamma:[0,1]\rightarrow [a,b]$ with $\theta=|\dot{\gamma}|=\mbox{L}(\gamma)\leq L$. 
We set $\kappa_{\gamma}=\kappa\circ\bar{\gamma}:[0,\theta]\rightarrow \mathbb{R}$. $\bar{\gamma}:[0,\theta]\rightarrow [a,b]$ denotes the unit speed reparametrization of $\gamma$. We use the convention $\infty\cdot 0=0$.
\smallskip
\item[(iv)] The statement in (iii) holds for any geodesic $\gamma:[0,1]\rightarrow [a,b]$.
\end{itemize}
\end{proposition}
\begin{proof}\textbf{1.}\ \
First, we prove that (iii) implies (i). 
Since $u$ is upper semi-continuous, it is bounded from above. Hence, $\sigma_{\kappa_{\gamma}}^{\sscr{(t)}}(\theta)=\infty$ implies $u\circ\gamma(1)=0$ for any geodesic $\gamma$. 
Therefore, one can find $L>L'>0$ such that that $\frs_{\kappa_{\gamma}}>0$ on $(0,\theta]$ 
for any constant speed geodesic $\gamma:[0,1]\rightarrow[a,b]$ with $\theta=|\dot{\gamma}|\leq L'$. Otherwise $u=const=0$.
$\frs_{\kappa_{\gamma}}>0$ implies $\sigma^{\sscr{(t)}}_{\kappa_{\gamma}}(\theta')<\infty$ for any $\theta'\in (0,\theta]$. 
\smallskip\\
\textit{Claim.} For $\kappa$ and $t$ fixed $f:h\mapsto \sigma_{\kappa}^{\sscr{(t)}}(h)$ is twice differentiable at $h=0$ and we have
\begin{align}\label{taylor}
h\in[0,{L}]\mapsto\sigma_{\kappa}^{(t)}(h)=t\left[1+\frac{1}{6}(1-t^2)\kappa(0)h^2\right]+o(h^2)_{\kappa}^{t}.
\end{align}
\textit{Proof of the claim:} 
We can compute the first and second derivative of $f$ at $0$ explicetly by application of l'Hosptital rule. Then we apply the Taylor expansion formula and the claim follows.  
\medskip
\\
If $\overline{\kappa}\geq \kappa\geq \underline{\kappa}$, then
\begin{align*}
o(h^2)_{{\kappa}}^{t}&=\sigma_{{\kappa}}^{(t)}(h)-\frac{1}{3}t(1-t^2){\kappa}(0)h^2-t\\
&\leq \sigma_{\overline{\kappa}}^{(t)}(h)-t\left[\frac{1}{3}(1-t^2)\underline{\kappa}h^2+1\right]=t\frac{1}{3}(1-t^2)(\overline{\kappa}-\underline{\kappa})h^2+o(h^2)^t_{\overline{\kappa}}
\end{align*}
and similar
$$
o(h^2)_{{\kappa}}^{t}\geq t\frac{1}{3}(1-t^2)(\underline{\kappa}-\overline{\kappa})h^2+o(h^2)^t_{\underline{\kappa}}.
$$
Since $\kappa$ is uniformly continuous on $[a,b]$, we can choose $\overline{h}>0$ and $(r_i)_{i=1,\dots,N}$ such that $$\max \kappa|_{[r_i-h,r_i+h]}-\min\kappa|_{[r_i-h,r_i+h]}<\epsilon$$ for each $i=1,\dots N$ and each $h\in[0,\overline{h}]$.
\smallskip\\
Upper semi-continuity of $u$ together with the condition (\ref{kuconcavity}) yiels continuity of $u$ on $[a,b]$. 
We consider $s\in[a,b]$, $h>0$ and a geodesic $\gamma:[0,1]\rightarrow [a,b]$ such that $\gamma_0=s-h$, $\gamma_1=s+h$ and $\gamma_{1/2}=s$ and $s\pm h\in [r_i-\underline{\kappa},r_i+\overline{\kappa}]$ for some $i=1,\dots N$.
Then, from
(\ref{taylor}) and (\ref{kuconcavity}) it follows that
\begin{align*}
&\frac{2u(s)-u(s-h)-u(s+h)}{h^2}\\
&\hspace{1cm}\geq \underbrace{\frac{\kappa(s-h)u(s-h)+\kappa(s+h)u(s+h)}{2}}_{\rightarrow\kappa(s)u(s)}-\epsilon+\underbrace{\frac{\min_{i=1,\dots,N}o(h^2)^t_{\min\kappa|_{[r_i-h,r_i+h]}}}{h^2}}_{\rightarrow 0}.
\end{align*}
Multiplication with $\phi\in C^{\infty}_0((a,b))$ such that $\phi\geq 0$, integration with respect to $s$, a change of variables and taking the limit $h\rightarrow 0$ yields
\begin{align*}
\int u(s)\phi''(s)ds\leq -\int\kappa(s)u(s)\phi(s)ds + \epsilon \int\phi(s)ds.
\end{align*}
Since $\epsilon>0$ can be choosen arbitrarily small, we obtain the result.
\smallskip
\\
\textbf{2.} We prove the equivalence between (i) and (ii). We assume (i) holds. Consider $v(t)=\int_0^1g(t,s)\kappa(\gamma(s))\theta^2 u(\gamma(s)) ds$. Then $v$ solves
\begin{align*}
v''(t)=-\kappa(\gamma(s))\theta^2 u(\gamma(s)) 
\end{align*}
in distributional sense by definition of the Green function. Hence, $u\circ\gamma -v$ has non-positive derivative in the distributional sense, and it follows that
$u\circ\gamma-v$ is concave (see Theorem 1.29 in \cite{simonconvexity}). This implies (ii). The backwards direction is straightforward and works like in the previous step.
\smallskip
\\
\textbf{3.}\ \
We prove that (i) implies (iv). The implication (iv) $\Rightarrow$ (iii) is obvious. First, we assume that $u\in C([a,b])\cap C^2((a,b))$. 
We consider the case when $\frs_{\kappa_{\gamma}}>0$ for any constant speed geodesic $\gamma:[0,1]\rightarrow (a,b)$.
The right-hand-side of (\ref{kuconcavity}) is denoted by $v(t)$ where $t\in[0,1]$. It is positive for any $t$ and solves $v''+\kappa_{\gamma}\circ\gamma\ \theta^2v=0$ 
with boundary condition $v(0)=u(\gamma(0))$ and $v(1)=u(\gamma(1))$.
Hence, it suffices to check that $\frac{u\circ{\gamma}}{v}$ has no local minimum in $(0,1)$. 
Otherwise, there is $\tau\in (0,1)$ such that $(\frac{u\circ{\gamma}}{v})'(\tau)=0$ and $(\frac{u\circ{\gamma}}{v})''(\tau)\geq 0$. 
We can deduce a contradiction exactly like in the proof of Proposition \ref{monotonicity}.
\smallskip

Next, we consider when 
there is a a constant speed geodesic $\gamma:[0,1]\rightarrow (a,b)$ such that $\frs_{\kappa_{\gamma}}(t_0)=0$ for some $t_0\in (0,\theta]$.
Again we adapt parts of the proof of Theorem 14.28 in \cite{viltot}.
We show that $u=0$.
Let $v(t)=\frs_{\kappa_{\gamma}}(\gamma(t))$ and $w(t)=u\circ\gamma(t)$. $v$ satisfies $v''+\kappa_{\gamma}\circ\gamma\theta^2v=0$ and $w$ satisfies $w''+\kappa_{\gamma}\circ\gamma\theta^2\leq 0$.
Consider $\frac{w}{v}=:h$. Then
\begin{align*}
(h'v^2)'&=h''v^2+2vv'h'=\left(\frac{w'v-v'w}{v^2}\right)'v^2+2vv'h'\\
&=\frac{w''v-v''w -v'w'+2(v')^2uw}{v^2}+\frac{2v'w'v^2-2(v')^2vw}{v^2}\\
&\leq-\kappa\theta^2\frac{w}{v}+\kappa\theta^2\frac{w}{v}=0
\end{align*}
Hence, $h'v^2$ is non-increasing. Suppose there is $\tau\in[0,1]$ such that $h'(\tau)>0$ then we also have that $h'v^2(\tau)>0$ and $h'v^2\geq C >0$ on $[\tau,1]$. 
for some constant $C>0$. Hence
$h'\geq C\frac{1}{v^2}$. $v=\frs_{\kappa_{\gamma}}\circ\gamma$ is in $C^2([0,1])$. Especially, it follows that $v(\delta)=\delta + o(\delta^2)$. Thus, $h'(h)\geq C \frac{1}{\delta^2}$. It follows
\begin{align*}
\int_{\delta}^{\epsilon}h'(\tau)d\tau=h(\epsilon)-h(\delta)\geq{C}\int_{\delta}^{\epsilon}\frac{1}{\tau^2}d\tau\rightarrow \infty \ \ \mbox{ if }\ \delta\rightarrow 0.
\end{align*}
Hence $h(\delta)\rightarrow -\infty$ if $\delta\rightarrow 0$
which contradicts $h\geq 0$. On the other hand, if there is $\tau\in [0,1]$ such that $h'(\tau)<0$, the same argument yields $h(\epsilon)\rightarrow -\infty$ if $\delta\rightarrow 0$.
It follows that $h'=0$ and $w(t)=c\cdot \frs_{\kappa_{\gamma}}(\gamma(t))$. Especially $u$ is differentiable at $\gamma(1)\in(a,b)$ with $u|_{(\gamma(0),\gamma(1))}>0$, $u(\gamma(1))=0$ and $u'(\gamma(1))\neq 0$ if $u\neq 0$
since $u'(\gamma(1))=0$ would contradict the uniquness of the solution of (\ref{ode}). But $u(\gamma(1))=0$ and $u'(\gamma(1))\neq 0$ yields $u(x)<0$ for $x\geq \gamma(1)$ which is not possible. Hence, $u=0$ and (\ref{kuconcavity}) holds.
\smallskip

Now, let $u$ be just upper semi-continuous. The equivalence between (i) and (ii) yields that $u$ is continuous. 
Consider $\phi\in C^{\infty}_0((0,1))$ with $\int_{0}^1\phi(t)dt=1$ and $\phi_{\epsilon}(t)=\frac{1}{\epsilon}\phi(\frac{t}{\epsilon})$. $\phi_{\epsilon}\in C^{\infty}_0((0,\epsilon))$.
We set
\begin{align*}
\tilde{u}(s)=u\star\phi_{\epsilon}(s)=\int_{-\epsilon}^0\phi_{\epsilon}(-r)u(s-r)dr=\int_{a}^b\phi_{\epsilon}(t-s)u(t)dr
\end{align*}
for $s\in [a,c]$ with $c<b$ such that $c+\epsilon\geq b$ and $\epsilon>0$ sufficiently small. 
$\kappa$ is uniformily continuous on $[a,b]$. Hence, for $\delta>0$ we can find $\bar{\epsilon}>0$ such that for all $\epsilon<\bar{\epsilon}$ we have $\kappa(s-r)\leq \kappa(s)+\delta$. Then
\begin{align*}
\tilde{u}''(s)=u\star \phi_{\epsilon}''(s)&=\int_a^b(\phi_{\epsilon}(t-s))''u(t)dt=\int_a^b\phi_{\epsilon}''(t-s)u(t)dt\\
&\leq -\int_{a}^b\phi_{\epsilon}(-r)\kappa(s-r)u(r-s)dr\leq -(\kappa(s)+\delta) \tilde{u}(s).
\end{align*}
Since $\tilde{u}\in C^2((a,c))\cap C^0([a,c])$ the previous conclusion holds for $\tilde{u}$ and $\tilde{\kappa}=\kappa+\delta$.
Now, since $u$ is continuous, $\tilde{u}\rightarrow u$ with respect to uniform convergence on $[a,c]$. And since solutions of (\ref{ode}) change uniformily continuous if the coefficient
$\kappa$ changes uniformily continuous on $[a,c]$, we obtain that $\frs_{\tilde{\kappa}_{\gamma}}\rightarrow \frs_{\kappa_{\gamma}}$ where $\gamma$ is a geodesic in $(c,b)$. 
Hence, in the case that where $\frs_{\kappa_{\gamma}}>0$ for any constant speed geodesic $\gamma:[0,1]\rightarrow (a,b)$, we obtain that $\frs_{\tilde{\kappa}_{\gamma}}>0$ for any constant speed geodesic $\gamma$ in $(a,c)$ by Sturm's comparison theorem. It follows that (\ref{ode}) holds for $\tilde{u}:[a,c]\rightarrow [0,\infty)$ and by uniform convergence it also holds for $u|_{[a,c]}$ if $\epsilon\rightarrow 0$. Then, it holds for $u$ since $c$ can be chosen arbitrarily close to $b$.
\smallskip

Finally, consider the case when there is a geodesic $\gamma$ in $(a,b)$ such that $\frs_{\kappa_{\gamma}}(\gamma(1))=0$. Then we can choose $c$ sufficiently close to $b$ and $\epsilon>0$ sufficiently small such that
there is a geodesic $\tilde{\gamma}$ in $(a,c)$ with
$\frs_{\tilde{\kappa}_{\gamma}}(\gamma(1))=0$. By the previous steps it follows that $\tilde{u}=\phi_{\epsilon}\star u=0$ that implies $u=0$.
\smallskip
%
\end{proof}
\paragraph{\textbf{$\kappa u$-concavity in metric spaces.}}
We consider a metric space $(X,\de_{\sX})$ and a lower semi-continuous function $\kappa:X\rightarrow \mathbb{R}$.
We define continuous functions $\kappa_n:X\rightarrow \mathbb{R}$ 
in the following way 
$$\kappa_n(x)=\inf_{y\in X}\left\{\kappa(y)+n\de_{\sX}(x,y)\right\}\leq \kappa(x).$$
We keep this notation for the rest of the article. $\kappa_n$ is monotone non-decreasing and converges pointwise to $\kappa$ as $n\rightarrow \infty$. 
For each $\kappa_n$ and for each Lipschitz curve $\gamma\in\mathcal{LC}(X)$ we can consider $\frs_{\kappa_{n,\gamma}}$ where $\kappa_{n,\gamma}=\kappa_n\circ\bar{\gamma}$
and $\bar{\gamma}:[0,\mbox{L}(\gamma)]\rightarrow X$ is the 1-speed reparametrization of $\gamma$.
If $\frs_{\kappa_{n,\gamma}}>0$ for all $n$, the generalized $\sin$-function $\frs_{\kappa_{n,\gamma}}$ is monotone non-increasing with respect to $n$. 
Hence, the limit exists pointwise everywhere in $[0,\mbox{L}(\gamma)]$. It is again denoted with ${\frs}_{\kappa_{\gamma}}$. 
$\frs_{\kappa_{\gamma}}$ is upper semi-continuous and if $\kappa$ is continuous, $\frs_{\kappa_{\gamma}}$
coincides with the previous definition.
This follows since $\kappa_{n,\gamma}$ converges uniformly to $\kappa_{\gamma}$ by Dini's theorem. 
Therefore, the stability of solutions of (\ref{ode}) under uniform changes of the coefficient $\kappa_{\gamma}$ implies that $\frs_{\kappa_{n,\gamma}}$ converges uniformily 
to the solution of (\ref{ode}) with coefficient $\kappa_{\gamma}$.
We can see that $\mathfrak{s}_{\kappa_{\gamma}}\geq \mathfrak{s}_{{\kappa}_{\gamma}'}$ if $\kappa,\kappa':X\rightarrow \mathbb{R}$ 
are lower semi-continuous and ${\kappa}'\geq {\kappa}$. In particular, we can consider $X=[a,b]\subset \mathbb{R}$.
\begin{definition}
Let $\kappa:X\rightarrow \mathbb{R}$ be lower semi-continuous and let $\gamma:[0,1]\rightarrow X$ be in $\mathcal{LC}(X)$ with $|\dot{\gamma}|=\theta$. 
Consider sequence $\kappa_n$ from above. Then $\sigma_{\kappa_{n,\gamma}}^{\sscr{(t)}}(\theta)$ is 
monotone non-decreasing in $\mathbb{R}\cup\left\{\infty\right\}$. 
We define the \textit{distortion coefficient with respect to $\kappa:X\rightarrow \mathbb{R}$ along $\gamma$} as
\begin{align*}
\sigma_{\kappa_{\gamma}}^{\sscr{(t)}}(\theta):=\lim_{n\rightarrow \infty}\sigma_{\kappa_{n,\gamma}}^{\sscr{(t)}}(\theta)\in \mathbb{R}\cup\left\{\infty\right\}.
\end{align*}
If $\kappa$ is continuous, the definition is consistent with the previous one. That is $\sigma_{\kappa_{\gamma}}^{\sscr{(t)}}(\theta)$ equals 
$\sigma_{\kappa\circ{\bar{\gamma}}}^{\sscr{(t)}}(\theta)$ 
as in Definition \ref{distort}. 
\end{definition}
\begin{lemma}
Let $\kappa:X\rightarrow \mathbb{R}$ be lower semi-continuous, and let $\gamma\in \mathcal{LC}(X)$ with $|\dot{\gamma}|=\theta$.
If $\sigma_{\kappa_{\gamma}}^{\sscr{(t_0)}}(\theta)=\infty$ for some $t_0\in (0,1)$ then $\sigma_{\kappa_{\gamma}}^{\sscr{(t)}}(\theta)=\infty$ for any $t\in (0,1)$.
\smallskip\\
In particular, either one has $\sigma_{\kappa_{\gamma}}^{\sscr{(t)}}(\theta)<\infty$ for any $t\in (0,1)$ and 
\begin{align*}
\sigma_{\kappa_{\gamma}}^{\sscr{(t)}}(\theta)={\frs_{\kappa_{\gamma}}(t\theta)}/{\frs_{\kappa_\gamma}(\theta)}\ \ \mbox{
where $\frs_{\kappa_{\gamma}}(\theta)\neq 0$, or $\sigma_{\kappa_{\gamma}}^{\sscr{(\cdot)}}(\theta)\equiv\infty$.}
\end{align*}
\end{lemma} 
\begin{proof}For the proof we write $\kappa_{n,\gamma}=\kappa_{n}$ and $\kappa_{\gamma}=\kappa$.
Assume $\sigma_{\kappa_n}^{\sscr{(t)}}(\theta)<\infty$ for each $n\in\mathbb{N}$. 
Otherwise, there is nothing to prove. Then, we must have that $\frs_{\kappa_n}(t_0\theta)/\frs_{\kappa_n}(\theta)\rightarrow \infty$.
Let $\underline{\kappa}=const\leq \kappa_n$ for all $n$. Hence, $\frs_{\kappa_n}(t\theta)\geq \frs_{\underline{\kappa}}(t\theta)$.
By proposition (\ref{central}) we have that
\begin{align*}
\frs_{\underline{\kappa}}(st_0\theta)\geq \sigma_{\underline{\kappa}}^{\sscr{(s)}}(t_0\theta)\frs_{{\kappa}_n}(t_0\theta)
\end{align*}
and
\begin{align*}
\frs_{\underline{\kappa}}(((1-s)t_0+s)\theta)\geq \sigma_{\underline{\kappa}}^{\sscr{(1-s)}}(t_0\theta)\frs_{{\kappa}_n}(t_0\theta)+\sigma_{\underline{\kappa}}^{\sscr{(s)}}(t_0\theta)\frs_{{\kappa}_n}(\theta)
\end{align*}
Hence, if we pick $t\in (0,1)$, we can write $t=st_0$ or $t=(1-s)t_0+s$. If $t=st_0$, we have the following estimate:
\begin{align*}
 \frs_{\kappa_n}(t\theta)/\frs_{\kappa_n}(\theta)\geq \sigma_{\underline{\kappa}}^{\sscr{(s)}}(t_0\theta)\underbrace{{\frs_{{\kappa}_n}(t_0\theta)}/\frs_{\kappa_n}(\theta)}_{\rightarrow \infty}.
\end{align*}
Similar for $t=(1-s)t_0+s$. Thus, $\sigma_{\kappa_n}^{(t)}(\theta)\rightarrow \infty$ for each $t\in (0,1)$ if $n\rightarrow \infty$.
\end{proof}
\begin{corollary}
Let $\kappa:X\rightarrow \mathbb{R}$ be lower semi-continuous, $\gamma$ is a geodesic in $X$.
Then $\kappa\mapsto \sigma_{\kappa_{\gamma}}^{\sscr(t)}(\theta)$ is monotone non-decreasing in the sense of Proposition \ref{monotonicity}.
\end{corollary}
\begin{proof}
If $\kappa'\geq \kappa$, let $\kappa'_n$ and $\kappa_n$ be the corresponding approximations. 
It is clear from the definition that $\kappa_{n,\gamma}'\geq \kappa_{n,\gamma}$. Hence, $\sigma_{\kappa_{n,\gamma}'}^{(t)}(\theta)\geq \sigma_{\kappa_{n,\gamma}}^{(t)}(\theta)$. Taking the limit $n\rightarrow \infty$ yields the result.
\end{proof}
\begin{remark}
If $\gamma \in \mathcal{LC}(X)$, we define $\gamma^-(t)=\gamma(1-t)$, 
and we set $$\sigma_{\kappa_{\gamma}^-}^{\sscr{(t)}}(\theta)=\sigma_{\kappa_{\gamma^-}}^{\sscr{(t)}}(\theta).$$
Therefore, one can see again that $\sigma_{\kappa}^{\sscr{(t)}}(\theta)=\infty$ if and only if $\sigma_{\kappa^-}^{\sscr{(t)}}(\theta)=\infty$.
\end{remark}
\begin{corollary}\label{central2}
Let $\kappa:X\rightarrow \mathbb{R}$ be lower semi-continuous, and let $u:X\rightarrow \mathbb{R}_{\geq 0}$ be upper semi-continuous. 
Then the following statements are equivalent:
\begin{itemize}
 \item[(i)]$(u\circ \bar{\gamma})''+\kappa_\gamma u\circ\bar{\gamma}\leq 0$ in the distributional sense for any constant speed geodesic $\gamma:[0,1]\rightarrow X$.
\smallskip
 \item[(ii)] There is a constant $0<L\leq b-a$ such that
 \begin{align*}
u(\gamma(t))\geq \sigma^{\sscr{(1-t)}}_{\kappa^{\sscr{-}}_{\gamma}}(\theta)u(\gamma(0))+\sigma^{\sscr{(t)}}_{\kappa^{\sscr{+}}_{\gamma}}(\theta)u(\gamma(1))
 \end{align*}
for any constant speed geodesic $\gamma:[0,1]\rightarrow X$ with $\theta=|\dot{\gamma}|=\mbox{L}(\gamma)\leq L$. 
\item[(iii)] The statement in (ii) holds for any geodesic $\gamma:[0,1]\rightarrow X$.
\end{itemize}
\end{corollary}
\begin{proof} If $\kappa$ is continuous, the result follows from Proposition \ref{central}.
If $\kappa$ is lower semi-continuous, we consider $\kappa_n$ for $n\in \mathbb{N}$. 
\smallskip\\
(ii) $\Rightarrow$ (i):
Since $\kappa_n\uparrow \kappa$, we have $\sigma_{\kappa_{n,\gamma}}^{\sscr{(t)}}(\theta)\uparrow\sigma_{\kappa_{\gamma}}^{\sscr{(t)}}(\theta)$ for $t\in (0,1)$. 
Then we can apply part \textbf{1.} of the proof of Proposition \ref{central} to obtain (\ref{distributional}) for $u$ with $\kappa$ replaced by $\kappa_n$.  
That is
\begin{align*}
-\int\phi''(t)u(t)dt&\geq\int\phi(t)\kappa_{n,\gamma}(t)u(t)dt\\
&=\underbrace{\int[\phi(t)\kappa_{n,\gamma}(t)u(t)]_+dt}_{\nearrow}-\underbrace{\int[\phi(t)\kappa_{n,\gamma}(t)u(t)]_-dt}_{\leq C<\infty}.
\end{align*}
for any $\phi\in C^{\infty}_c((0,|\dot{\gamma}|))$ where the left hand side and $C$ are independent of $n$. Hence, the right hand side converges to the integral of $\phi\kappa_{\gamma}u$.
\smallskip
\\
(i) $\Rightarrow$ (iii):  We can apply part \textbf{3.} from the proof of Proposition \ref{central}, and obtain (\ref{kuconcavity}) with $\kappa$
replaced by $\kappa_n$. By the definition of distortion coefficients for general $\kappa$ the result follows.
\end{proof}


\begin{lemma}\label{gr}
Consider $\lambda\in [0,1]$, $\theta>0$, a curve $\gamma\in \mathcal{LC}(X)$ with $\mbox{L}(\gamma)=\theta$ and $\kappa,\kappa':X\rightarrow \mathbb{R}$ lower semi-continuous. Then
\begin{align*}
\sigma^{\sscr (t)}_{\kappa_{\gamma}}(\theta)^{1-\lambda}\cdot\sigma^{\sscr (t)}_{\kappa_{\gamma}'}(\theta)^{\lambda}\geq \sigma^{\sscr (t)}_{(1-\lambda)\kappa_{\gamma}+\lambda\kappa_{\gamma}'}(\theta).
\end{align*}
Especially, $\kappa\mapsto \log\sigma_{\kappa_{\gamma}}$ is convex.
\end{lemma}
\begin{proof} For the proof we write $\kappa_{n,\gamma}=\kappa_{n}$ and $\kappa_{\gamma}=\kappa$.
Assume $\sigma_{\kappa}^{\sscr{(t)}}(\theta)<\infty$ and $\sigma_{\kappa'}^{\sscr{(t)}}(\theta)<\infty$ for each $t\in (0,1)$, since otherwise there is nothing to prove. 
We assume first that $\kappa$ and $\kappa'$ are continuous.
$l:\ t\mapsto\log\left[\sigma^{\sscr (t)}_{\kappa}(\theta)^{1-\lambda}\cdot\sigma^{\sscr (t)}_{\kappa'}(\theta)^{\lambda}\right]$ solves
\begin{align*}
l''\leq -(1-\lambda)\kappa-\lambda{\kappa'}-(l')^2.
\end{align*}
Hence
$\sigma^{\sscr (t)}_{\kappa}(\theta)^{1-\lambda}\cdot\sigma^{\sscr (t)}_{\kappa'}(\theta)^{\lambda}$ solves $v''+\big((1-\lambda)\kappa+\lambda{\kappa}'\big)v\leq 0$ with boundary
condition $v(0)=0$ and $v(1)=1$.
The result follows by the previous theorem.
\smallskip

If $\kappa$ and $\kappa'$ are lower semi-continuous, we consider again their approximations by $\kappa_n$ and $\kappa'_n$. We easily obtain that
\begin{align*}
\sigma^{\sscr (t)}_{\kappa}(\theta)^{1-\lambda}\cdot\sigma^{\sscr (t)}_{\kappa'}(\theta)^{\lambda}\geq \sigma^{\sscr (t)}_{\kappa_n}(\theta)^{1-\lambda}\cdot\sigma^{\sscr (t)}_{\kappa'_n}(\theta)^{\lambda}\geq\sigma^{\sscr (t)}_{(1-\lambda)\kappa_n+\lambda\kappa'_n}(\theta).
\end{align*}
We show that $\sigma^{\sscr (t)}_{(1-\lambda)\kappa_n+\lambda\kappa'_n}(\theta)\rightarrow \sigma^{\sscr (t)}_{(1-\lambda)\kappa+\lambda\kappa'}(\theta)$.
One can check that $(1-\lambda)\kappa_n+\lambda\kappa'_n\leq ((1-\lambda)\kappa+\lambda\kappa')_n$.
On the other hand, by continuity of the approximating sequence for all $n\in \mathbb{R}$ and for all $x\in [0,\theta]$ there exists $m_{x}\geq 2^n$ and $\delta_x>0$ 
such that $(1-\lambda)\kappa_{\bar{m}}+\lambda\kappa'_{\bar{m}}\geq ((1-\lambda)\kappa+\lambda\kappa')_n$ 
on $B_{\delta_x}(x)$ for all $\bar{m}\geq m_x$. Hence, by compactness of $[0,\theta]$ 
we can choose $x_1,\dots,x_n$ such that $[0,\theta]\subset \bigcup_{i=1,\dots,n}B_{\delta_{x_i}}(x_i)$. Then 
$(1-\lambda)\kappa_{m_n}+\lambda\kappa'_{m_n}\geq ((1-\lambda)\kappa+\lambda\kappa')_n$ for $m_n:=\max_{i}m_{x_i}$. Hence, 
$$\underbrace{\sigma_{((1-\lambda)\kappa+\lambda\kappa')_{m_n}}^{\sscr{(t)}}(\theta)}_{\rightarrow \sigma_{(1-\lambda)\kappa+\lambda\kappa'}^{\sscr{(t)}}(\theta)}\leq \sigma_{(1-\lambda)\kappa_{m_n}+\lambda\kappa'_{m_n}}^{\sscr{(t)}}(\theta)\leq \underbrace{ \sigma_{((1-\lambda)\kappa+\lambda\kappa')_n}^{\sscr{(t)}}(\theta)}_{\rightarrow \sigma^{\sscr{(t)}}_{(1-\lambda)\kappa+\lambda\kappa'}(\theta)}.$$
Hence, $\sigma^{\sscr (t)}_{(1-\lambda)\kappa_n+\lambda\kappa'_n}(\theta)\rightarrow \sigma^{\sscr (t)}_{(1-\lambda)\kappa+\lambda\kappa'}(\theta)$.
\end{proof}
\begin{proposition}\label{contin}
Let $\kappa:X\rightarrow \mathbb{R}$ be continuous (lower semi-continuous). Let $t\in (0,1)$. Then the map
\begin{align*}
\gamma\in (\mathcal{LC}(X),\de_{\infty}) \mapsto \sigma_{\kappa^{+/-}_{\gamma}}^{\sscr{(t)}}(|\dot{\gamma}|)\in\mathbb{R}\cup\left\{\infty\right\}
\end{align*}
is continuous (lower semi-continuous).
\end{proposition}
\begin{proof}
If $\kappa$ is continuous, the result follows from Proposition \ref{continuity}.
For $\kappa$ lower semi-continuous we consider its continuous approximation $\kappa_n$. Then by definition for any Lipschitz curve $\gamma\in \mathcal{LC}(X)$
$$\sigma_{\kappa^{+/-}_{n,\gamma}}^{\sscr{(t)}}(|\dot{\gamma}|)\uparrow \sigma_{\kappa^{+/-}_{\gamma}}^{\sscr{(t)}}(|\dot{\gamma}|).$$
In particular, $\gamma\mapsto\sigma_{\kappa^{+/-}_{\gamma}}^{\sscr{(t)}}(|\dot{\gamma}|)$ is lower semi-continuous
\end{proof}
\begin{definition}\label{ahr}
Consider a metric space $(Y,\de_{\sY})$ and a lower semi-continuous function $\kappa:Y\rightarrow \mathbb{R}$. 
We say a function $u:Y\rightarrow [0,\infty)$ is \textit{$\kappa u$-convex} if 
$u<\infty$ and 
for all geodesics $\gamma:[0,1]\rightarrow Y$
\begin{align}\label{grugru}
u(\gamma(t))\geq \sigma^{\sscr{(1-t)}}_{\kappa^{\sscr{-}}_{\gamma}}(\mbox{L}(\gamma))u(\gamma(0))+\sigma^{\sscr{(t)}}_{\kappa^{\sscr{+}}_{\gamma}}(\mbox{L}(\gamma))u(\gamma(1))
\end{align}
where $\kappa_{\gamma}=\kappa\circ\bar{\gamma}:[0,\mbox{L}(\gamma)]\rightarrow Y$ and $\bar{\gamma}$ is the unit speed reparametrization of $\gamma$.
\smallskip\\
We say $u$ is weakly $\kappa u$-convex if $u<\infty$ and for all $x,y\in Y$ there exists a geodesic $\gamma:[0,1]\rightarrow Y$ between $x$ and $y$ such that (\ref{grugru}) holds.
\smallskip\\
We say a function $f:Y\rightarrow \mathbb{R}\cup\left\{\pm\infty\right\}$ is (weakly) $(\kappa,N)$-convex if $e^{-\frac{f}{N}}=u$ is (weakly) $\frac{\kappa}{N} u$-concave. 
We use the convention $e^{\infty}=\infty$, $e^{-\infty}=0$.
\end{definition}
\section{Curvature-dimension condition}
Let $(X,\de_{\sX},\m_{\sX})$ be a metric measure space.
Given a number $N\in \mathbb{R}$ with $N\geq 1$, we define the \textit{N-R\'enyi entropy functional}
\begin{align*}
S_N(\ \cdot\ |\m_{\sX}):\mathcal{P}_2(X)\rightarrow \mathbb{R}
\end{align*}
with respect to $\m_{\sX}$ by 
\begin{align*}
\nu=\varrho\m_{\sX}+\nu^s\mapsto S_N(\nu):=S_N(\nu|\m_{\sX}):=-\int_X\varrho^{\frac{1}{N}}d\nu
\end{align*}
where $\varrho\m_{\sX}+\nu^s$ is the Lebesgue decomposition of $\nu$. If $\m_{\sX}$ is a finite measure
for each $\nu\in \mathcal{P}_2(X)$ we have
\begin{align*}
\mbox{Ent}(\nu|\m_{\sX})=\lim_{N\rightarrow \infty}N(1+S_{N}(\nu)),
\end{align*}
where $\mbox{Ent}$ is the \textit{Boltzmann-Shanon entropy functional}. 
\smallskip
\\
\\
We consider $\kappa={\VK}/{N}$ where $\VK:X\rightarrow \mathbb{R}$ is lower semi-continuous and locally bounded from below, 
and we set $\sigma_{\VK_{\gamma}/N}^{\sscr (t)}(\theta)=\sigma_{\theta^2\VK_{\gamma}(\cdot\theta)/N}^{\sscr (t)}(1)=\sigma_{\VK_{\gamma},N}^{\sscr (t)}(\theta)$ 
where $\gamma\in \mathcal{LC}(X)$ and $\theta=|\dot{\gamma}|$.
\begin{definition}
Let $(X,\de_{\sX},\m_{\sX})$, $\VK$ and $\gamma$ as before. We define \textit{generalized distortion coefficients with respect to $\VK$ and $N$ along $\gamma$} as
\begin{align*}
\tau_{\VK_{\gamma},N}^{(t)}(\theta)=\begin{cases}\theta\cdot\infty&\mbox{ if }\VK>0\mbox{ and }N=1\\
t^{\frac{1}{\sN}}\big[\sigma_{\VK_{\gamma},N-1}^{\sscr (t)}(\theta)\big]^{\frac{\sN-1}{\sN}}&\mbox{ otherwise.}
\end{cases}
\end{align*} 
We use the conventions $r\cdot\infty=\infty$ for $r>0$, $0\cdot\infty=0$ and $(\infty)^{\alpha}=\infty$ for $\alpha> 0$. If 
$\VK>0$, we have $\tau^{(t)}_{\sk_{\gamma},1}(\theta)<\infty$ if and only if $\theta=0$, and $\tau_{\sk_{\gamma},1}^{(t)}(\theta)=t$ if $\VK\leq 0$.
\end{definition}
\begin{corollary}\label{somethingmore}
For $\VK,\VK':[0,1]\rightarrow \mathbb{R}$, $N,N'>0$, $t\in[0,1]$ and $\theta>0$,
\begin{align*}
\sigma_{\VK,N}^{\sscr (t)}(\theta)^{N}\sigma_{\VK',N'}^{\sscr (t)}(\theta)^{N'}\geq \sigma_{\VK+\VK',N+N'}^{\sscr (t)}(\theta)^{N+N'}
\end{align*}
and, if $N\geq 1$, 
\begin{align*}
\tau_{\VK,N}^{\sscr (t)}(\theta)^{N}\sigma_{\VK',N'}^{\sscr (t)}(\theta)^{N'}\geq \tau_{\VK+\VK',N+N'}^{\sscr (t)}(\theta)^{N+N'}, 
\end{align*}
and in particular
$$
\tau_{\VK,N}^{\sscr (t)}(\theta)^{N}\tau_{\VK',N'}^{\sscr (t)}(\theta)^{N'}\geq \tau_{\VK+\VK',N+N'}^{\sscr (t)}(\theta)^{N+N'}.
$$
\end{corollary}
\proof The result follows directly from Lemma \ref{gr}.
\begin{remark}
For the rest of the article we always assume that $(X,\de_{\sX},\m_{\sX})$ is a metric measure space and $\VK:X\rightarrow \mathbb{R}$ 
is lower semi-continuous and locally bounded from below. 
In this case we say that $\VK$ is an admissible function. 
It follows from Proposition \ref{contin} that if $\VK$ is continuous (lower semi-continuous), the map
$$\gamma\in\mathcal{G}(X)\mapsto \tau_{\sk^{+/-}_{\gamma},\sN}^{\sscr{(t)}}(|\dot{\gamma}|)\in \mathbb{R}_{\geq 0}\cup\left\{\infty\right\}$$ 
is continuous (lower semi-continuous) for $t\in [0,1]$. In particular, it is measurable and we can integrate it with respect to probability measures on $\mathcal{G}(X)$.
\end{remark}
\begin{definition}\label{bigg}
Consider an admissible function $\VK:X\rightarrow \mathbb{R}$, and let $N\in\mathbb{R}$ with $N\geq 1$.
$(X,\de_{\sX},\m_{\sX})$ satisfies the \textit{curvature-dimension condition}
$CD(\VK,N)$ if for each pair $\nu_0,\nu_1\in \mathcal{P}_2(X,\m_{\sX})$ with bounded support
there exists a dynamical optimal coupling $\Pi$ of $\nu_0=\varrho_0d\m_{\sX}$ and $\nu_1=\varrho_1d\m_{\sX}$ and a geodesic $(\nu_t)_{t\in[0,1]}\subset\mathcal{P}_2(X,\m_{\sX})$, such that
\begin{align}\label{curvaturedimension}
S_{N'}(\nu_t)\leq{\textstyle -}\!\!\int\Big[\tau_{\VK^{\sscr{-}}_{\gamma},N'}^{\sscr{(1-t)}}(|\dot{\gamma}|)\varrho_0\left(e_0(\gamma)\right)^{-\frac{1}{N'}}+\tau_{\VK^{\sscr{+}}_{\gamma},N'}^{\sscr{(t)}}(|\dot{\gamma}|)\varrho_1\left(e_1(\gamma)\right)^{-\frac{1}{N'}}\Big]d\Pi(\gamma)
\end{align}
for all $t\in [0,1]$ and all $N'\geq N$. $\VK_{\gamma}=\VK\circ\bar{\gamma}$ where $\gamma:[0,1]\rightarrow X$ is a geodesic and $\bar{\gamma}$ its 1-speed reparametrization.
The right hand side of (\ref{curvaturedimension}) is also denoted with $T^{\sscr{(t)}}_{\VK,N'}(\Pi|\m_{\sX})$.
\end{definition}
\begin{remark}
If $\Pi$ is the optimal dynamical coupling from the previous definition, let $\Pi'(x_0,x_1)(d\gamma)=:\Pi'_{x_0,x_1}(d\gamma)$ be its disintegration with respect to $(e_0,e_1)_{\star}\Pi=\pi$. 
One can reformulate (\ref{curvaturedimension}) in the following way
\begin{align}\label{gluck}
S_{N'}(\nu_t)\leq{\textstyle -}\!\!\int\Big[\mathcal{T}_{\VK^{\sscr{-}},N'}^{\sscr{(1-t)}}(\Pi'_{x_0,x_1})\varrho_0\left(x_0\right)^{-\frac{1}{N'}}+\mathcal{T}_{\VK^+,N'}^{\sscr{(t)}}(\Pi'_{x_0,x_1})\varrho_1\left(x_1\right)^{-\frac{1}{N'}}\Big]d\pi(x_0,x_1)
\end{align}
where 
\begin{align*}
\mathcal{T}_{\VK^-,N'}^{\sscr{(1-t)}}(\Pi')=\int\tau_{\VK^{\sscr{-}}_{\gamma},N'}^{\sscr{(1-t)}}(|\dot{\gamma}|)d\Pi'(d\gamma)
\end{align*}
for any measure $\Pi'\in\mathcal{G}(X)$. \\
Conversely, if there is a kernel $\Pi_{x_0,x_1}'(d\gamma)$ such that for $\mu_0$ and $\mu_1$ there exists a geodesic $\mu_t$ and an optimal coupling $\pi$ with (\ref{gluck}), then
$X$ satisfies $CD(\VK,N)$.
\end{remark}

\begin{remark}
In the case where $\VK$ is constant the previous definition is equivalent to Sturm's curvature-dimension condition in \cite{stugeo2}
since
a measurable selection theorem yields a measurable map $(x,y)\mapsto \gamma_{x,y}\in\mathcal{G}(X)$ where $\gamma_{x,y}$ is a geodesic between $x$ and $y$.
\end{remark}
\begin{definition}
Two metric measure space $(X,\de_{\sX},\m_{\sX})$ and $(X',\de_{\sX'},\m_{\sX'})$ are called isomorphic if there exists an isometry $\psi:\supp\m_X\rightarrow \supp\m_{X'}$ such that 
\begin{align*}
\psi_{\star}\m_{\sX}=\m_{\sX'}.
\end{align*}
\end{definition}

\begin{proposition}
Let $(X,\de_{\sX},\m_{\sX})$ be a metric measure space which satisfies the condition $CD(\VK,N)$ for a continuous function $\VK:X\rightarrow \mathbb{R}$ and $N\geq 1$. 
\begin{itemize}
\item[(i)]
If there is an isomorphism $\psi:(X,\de_{\sX},\m_{\sX})\rightarrow (X',\de_{\sX'},\m_{\sX'})$ onto a metric measure space $(X',\de_{\sX'},\m_{\sX'})$ then $(X',\de_{\sX'},\m_{\sX'})$
satisfies the condition $CD(\psi_{\star}\VK,N)$ with $\psi_{\star}\VK=\VK\circ\psi$.
\item[(ii)]
For $\alpha,\beta>0$ the rescaled metric measure space $(X',\alpha\de_{\sX'},\beta\m_{\sX'})$ satisfies $CD(\alpha^{-2}\VK,N)$.
\item[(iii)]
For each convex subset $U\subset X$ the metric measure space $(U,\de_{\sX}|_{U\times U},\m_{\sX}|_{U})$ satisfies $CD(\VK|_{U},N)$.
\end{itemize}
\end{proposition}
\begin{proof}
(i) First, we observe that $\psi^{\star}\VK$ is still lower semi-continuous and locally bounded from below. $\psi$ induces an isometry from $\mathcal{P}_2(X,\m_{\sX})$ to $\mathcal{P}_2(X',\m_{\sX'})$, and 
the image of a geodesic in $X$ is a geodesic in $X'$. Moreover,
\begin{align*}
\int_X\varrho_t^{-\frac{1}{N}+1}d\m_{\sX}=\int_{X'}(\varrho_t\circ\psi)^{-\frac{1}{N}+1}d\m_{\sX'}
\end{align*}
and $\psi_{\star}\Pi$ is an optimal dynamical transference plans provided $\Pi$ is so. Then result follows.
\smallskip

(ii), (iii) The results follow easily. One can easily adapt the proofs of similar statements in \cite{stugeo2}.
\end{proof}
In \cite{sturmvariable} Sturm gave the definition of the condition $CD(\VK,\infty)$ for a lower semi continuous function $\VK:X\rightarrow \mathbb{R}$.
\begin{definition}
$(X,\de_{\sX},\m_{\sX})$ satisfies the condition $CD(\VK,\infty)$ if for any pair $\mu_0,\mu_1\in\mathcal{P}_2(X)$ there exists a $W_2$-geodesic $\mu_t$ and an optimal dynamical transference plan $\Pi$ such that 
$\mu_t=(e_t)_{\star}\Pi$ and 
\begin{align}\label{nacht}
\Ent(\mu_t)\leq (1-t)\Ent(\mu_0)+t\Ent(\mu_1)-\int_0^1\int_{\mathcal{G}(X)}g(s,t) \VK(\gamma(s))|\dot{\gamma}(s)|^2d\Pi(\gamma)ds
\end{align}
for all $t\in[0,1]$. $g(s,t)=\min\left\{s(1-t),t(1-s)\right\}$ is the \textit{Green function} of the unit interval and $\Ent:\mathcal{P}_2(X)\rightarrow \mathbb{R}\cup\left\{-\infty\right\}$ is the Shanon entropy functional.
\end{definition}
\begin{proposition}\label{htht}
Let $(X,\de_{\sX},\m_{\sX})$ be a metric measure space which satisfies the condition $CD(\VK,N)$ for a continuous function $\VK:X\rightarrow \mathbb{R}$ and $N\geq 1$. 
\begin{itemize}
\item[(i)] If $\VK':X\rightarrow \mathbb{R}$ is a continuous function such that $\VK'\leq \VK$, and if $N'\geq N$, then $(X,\de_{\sX},\m_{\sX})$ also satisfies the condition
$CD(\VK',N')$. 
\item[(ii)]
If $(X,\de_{\sX},\m_{\sX})$ has finite mass then it satisfies the condition $CD(\VK,\infty)$ in the sense of Sturm.
\end{itemize}
\end{proposition}
\begin{proof}
(i) is an immediate consequence of the monotonicity of $\sigma_{\kappa}^{\sscr{(t)}}(\theta)$ with respect to $\kappa$. 
\smallskip

For (ii) it suffices to consider $\nu_0,\nu_1\in \mathcal{P}_2(X,\m_{\sX})$ with $\Ent(\nu_0|m_{\sX})<\infty$ and $\Ent(\nu_1|m_{\sX})<\infty$. 
In any other case the right hand side in (\ref{nacht}) is $\infty$. By assumption, $(X,\de_{\sX},\m_{\sX})$ satisfies $CD(\VK,N)$. 
Hence, there exists a dynamical optimal transference plan $\Gamma$ betweem $\nu_0$ and $\nu_1$ such that (\ref{curvaturedimension}) is satisfied for $\forall N'\geq N$.

Since $\m_{\sX}(X)<\infty$ it implies that $\Ent((e_t)_*\Gamma|\m_{\sX})=\lim_{N'\rightarrow \infty}(1+S_{N'}((e_t)_*\Gamma|\m_{\sX})$ for $t\in [0,1]$. It follows
\begin{align*}
&N'(1+S_{N'}((e_t)_*\Gamma|\m_{\sX}))\\
&\leq-N'\int\left[-(1-t)+\tau_{\VK_{\gamma}^{\sscr{-}},N'}^{\sscr{(1-t)}}(|\dot{\gamma}|)\varrho_0(e_0(\gamma))^{-\frac{1}{N'}}-t+\tau_{\VK_{\gamma}^{\sscr{+}},N'}^{\sscr{(t)}}(|\dot{\gamma}|)\varrho_1(e_1(\gamma))^{-\frac{1}{N'}}\right]\Gamma(\gamma)\\
&\leq (1-t)N'(1+S_{N'}((e_0)_*\Gamma|\m_{\sX}))+tN'(1+S_{N'}((e_1)_*\Gamma|\m_{\sX}))\\
&-N'\int\left[\left[(1-t)+\sigma_{\VK_{\gamma}^{\sscr{-}},N'}^{\sscr{(1-t)}}(|\dot{\gamma}|)\right]\varrho_0(e_0(\gamma))^{\frac{-1}{N'}}+\left[t+\sigma_{\VK_{\gamma}^{\sscr{+}},N'}^{\sscr{(t)}}(|\dot{\gamma}|)\right]\varrho_1(e_1(\gamma))^{\frac{-1}{N'}}\right]\Gamma(\gamma)
\\
&\leq (1-t)N'(1+S_{N'}((e_0)_*\Gamma|\m_{\sX}))+tN'(1+S_{N'}((e_1)_*\Gamma|\m_{\sX}))\\
&-\int \underbrace{N'\left[(1-\sigma_{\VK_{\gamma}^{\sscr{-}},N'}^{\sscr{(1-t)}}(|\dot{\gamma}|)-\sigma_{\VK_{\gamma}^{\sscr{+}},N'}^{\sscr{(t)}}(|\dot{\gamma}|)\right]}_{=w(t)}\Gamma(\gamma)
\end{align*}
$w$ solves $w''=-\VK_{\gamma}|\dot{\gamma}|^2(\sigma_{\VK_{\gamma}^{\sscr{-}},N}^{\sscr{(1-t)}}(|\dot{\gamma}|)+\sigma_{\VK_{\gamma}^{\sscr{+}},N}^{\sscr{(t)}}(|\dot{\gamma}|))$ with $w(0)=w(1)=0$. Hence
$$w= \int_0^1\left[g(s,t){\VK}_{\gamma}|\dot{\gamma}|^2(\sigma_{\VK_{\gamma}^{\sscr{-}},N}^{\sscr{(1-s)}}(|\dot{\gamma}|)+\sigma_{\VK_{\gamma}^{\sscr{+}},N}^{\sscr{(s)}}(|\dot{\gamma}|))\right]ds.$$
Since $\sigma_{\VK_{\gamma}^{\sscr{-}},N}^{\sscr{(1-t)}}(|\dot{\gamma}|)+\sigma_{\VK_{\gamma}^{\sscr{+}},N}^{\sscr{(t)}}(|\dot{\gamma}|)\rightarrow 1$ if $N'\rightarrow\infty$ uniformily in $\gamma\in\mathcal{G}(X)$ for fixed $t$, 
it follows
\begin{align*}
&N'(1+S_{N'}((e_t)_*\Gamma|\m_{\sX}))\\
&\leq (1-t)N'(1+S_{N'}((e_0)_*\Gamma|\m_{\sX}))+tN'(1+S_{N'}((e_1)_*\Gamma|\m_{\sX}))\\
&\hspace{1cm}-\int\int_0^1\left[g(s,t){\VK}_{\gamma}|\dot{\gamma}|^2(\sigma_{\VK_{\gamma}^{\sscr{-}},N}^{\sscr{(1-t)}}(|\dot{\gamma}|)+\sigma_{\VK_{\gamma}^{\sscr{+}},N}^{\sscr{(t)}}(|\dot{\gamma}|))\right]ds\Gamma(\gamma)\\
&\hspace{1cm}\left[\rightarrow -\int\int_0^1g(s,t){\VK}_{\gamma}|\dot{\gamma}|^2ds\Gamma(\gamma)\mbox{ if }N'\rightarrow \infty\right]
\end{align*}
and this implies the result.
\end{proof}
%
\begin{theorem}\label{smoothcase}
Let $(M,g_{\sM},Vd\vol_{\sM})$ be a weighted Riemannian manifold for a smooth function $V:M\rightarrow (0,\infty)$. Let $\VK:M\rightarrow \mathbb{R}$ be a lower semi-continuous function and $N\geq 1$. 
\\
The metric measure space $(M,\de_{\sM},Vd\vol_{\sM})$ satisfies the curvature-dimension condition $CD(\VK,N)$ if and only if $(M,g_{\sM},Vd\vol_{\sM}))$ has $N$-Ricci curvature bounded from below by $\VK$.
\end{theorem}
\begin{remark}
For each real number $N>n$ the $N$-Ricci tensor is defined as
\begin{align*}\ric^{N,V}(v)
&=\ric(v)-(N-n)\frac{\nabla^2 V^{\frac{1}{N-n}}(v)}{V^{\frac{1}{N-n}}(p)}\end{align*}
where $v\in TM_p$. For $N=n$ we define
$$\ric^{N,V}(v):=
\begin{cases}
\ric(v)
+\nabla^2\log V(v) &\nabla \log V (v)=0\\
-\infty & \mbox{else}.
\end{cases}$$
For $1\leq N<n$ we define $\ric^{N,\Psi}(v):=-\infty$ for all $v\neq 0$ and $0$ otherwise.

\end{remark}
\begin{example}
Let $\overline{(\alpha,\beta)}=I\subset \mathbb{R}$ be some interval  where $\alpha,\beta\in\mathbb{R}\cup\left\{\pm\infty\right\}$.
Let $\VK:I\rightarrow \mathbb{R}$ be a lower semi-continuous function and let $u:I\rightarrow [0,\infty)$ be a non-negative solution
of
\begin{align*}
u''+ \VK u=0.
\end{align*}
Then for any $N\geq 1$, the metric measure space
\begin{align*}
(I,|\cdot|_2,u^{N-1}d\mathcal{L}^1)
\end{align*}
satisfies the curvature-dimension $CD(\VK,N)$.
\end{example}

\begin{proof}
``$\Leftarrow$'': Pick a point $p\in M$ and $\epsilon>0$ such that $\VK|_{B_{\epsilon}(p)}\geq \VK_{\epsilon}$. 
There exists geodesically convex ball $B_{\delta}(p)$ for $0<\delta<\epsilon$ around $p$. Hence, $$(B_{\delta}(p),d_{\sM}|_{B_{\delta}(p)},Vd\vol_{\sM}|_{B_{\delta}})$$
satisfies the condition $CD(\VK_{\epsilon},N)$. It follows that the $N$-Ricci tensor is bounded from below by $\VK_{\epsilon}$ (for instance see \cite{stugeo2}).
If $\epsilon$ goes to $0$, we see that $\VK_{\epsilon}\rightarrow \VK(p)$ and the result follows.\\
``$\Rightarrow$'':
The proof goes exactly as the proof of the corresponding result in \cite{stugeo2}, \cite{lottvillani} or \cite{CMS}.
\end{proof}
\section{Geometric consequences}
In this section we assume $\supp\m_{\sX}=X$. 
\begin{theorem}[Brunn-Minkowski inequality]\label{bm}
Assume that the metric measure space $(X,\de_{\sX},\m_{\sX})$ satisfies $CD(\VK,N)$ for $\VK:X\rightarrow \mathbb{R}$ lower semi-continuous and $N\geq 1$.
Let $A_0,A_1\subset X$ be bounded Borel sets with $\m_{\sX}(A_0)\m_{\sX}(A_1)>0$. 
Set $\mathcal{G}(A_0,A_1)=\left\{\gamma\in \mathcal{G}(X):\gamma(i)\in A_i, i=0,1\right\}$. Then
\begin{align}\label{gibb}
\m_{\sX}(A_t)^{\frac{1}{N}}\geq \inf_{\gamma\in \mathcal{G}(A_0,A_1)}\tau_{{\sk}^-_{\gamma},\sN}^{\sscr{(1-t)}}(|\dot{\gamma}|)\m_{\sX}(A_0)^{\frac{1}{N}}+ \inf_{\gamma\in \mathcal{G}(A_0,A_1)}\tau_{{\sk}^+_{\gamma},\sN}^{\sscr{(t)}}(|\dot{\gamma}|)\m_{\sX}(A_1)^{\frac{1}{N}}.
\end{align}
where $\inf_{\gamma\in \mathcal{G}(A_0,A_1)}\tau_{{\sk}^{-/+}_{\gamma},\sN}^{\sscr{(1-t/t)}}(|\dot{\gamma}|)\geq 0$.
\end{theorem}
\begin{proof}
First, assume $\m(A_0),\m(A_1)<\infty$ and set $\mu_i=\m(A_i)^{-1}\m|_{A_i}$ for $i=0,1$. The curvature-dimension yields
\begin{align*}
\int_{A_t}\varrho_t^{\frac{1}{N'}}d\mu_t\geq \int\tau_{{\sk}^-_{\gamma},\sN}^{\sscr{(1-t)}}(|\dot{\gamma}|)\m_{\sX}(A_0)^{1/N}+ \tau_{{\sk}^+_{\gamma},\sN}^{\sscr{(t)}}(|\dot{\gamma}|)\m_{\sX}(A_1)^{1/N}
\end{align*}
where $(\mu_t=\varrho_td\m_{\sX})_t$ denotes the absolutely continuous geodesic that connects $\mu_0$ and $\mu_1$, and $\Pi$ is an optimal dynamical plan. By Jensen's inequality the left hand side of
the previous inequality is smaller than $\m_{\sX}(A_t)^{\frac{1}{N'}}$. The general case follows by approximation of $A_i$ by sets of finite measure.
\end{proof}
\begin{definition}[Minkowski content]
Consider $x_0\in X$ and $B_r(x_0)\subset X$. Set $v(r)=\m_{\sX}(\bar{B}_r(x_0))$. The Minkowski content of $\partial B_r(x_0)$ (the $r$-sphere around $x_0$) is defined as
\begin{align*}
s(r):=\limsup_{\delta\rightarrow 0}\frac{1}{\delta}\m_{\sX}(\bar{B}_{r+\delta}(x_0)\backslash B_r(x_0)).
\end{align*}
\end{definition}

\begin{theorem}\label{useful}
Assume $(X,\de_{\sX},\m_{\sX})$ satisfies $CD(\VK,N)$ for an admissible function $\VK$ and $N\in[1,\infty)$.
Then, $(X,\de_{\sX})$ is a proper metric space, each bounded set has finite measure and satisfies a doubling property, and either $\m_{\sX}$ is supported by one point or all points and all sphere have mass $0$. 

In particular, if $N>1$ then for each $x_0\in X$, for all $0<r<R$ and $\underline{\VK}\in\mathbb{R}$ such that $\VK|_{B_R(x_0)}\geq \underline{\VK}$ and $R\leq \pi\sqrt{(N-1)/\underline{\VK}\vee 0}$, we have
\begin{align}\label{one}
\frac{s(r)}{s(R)}\geq \frac{\sin_{\underline{\sk}/(\sN-1)}^{\sN-1}r}{\sin_{\underline{\sk}/(\sN-1)}^{\sN-1}R} \ \ \&\ \ 
\frac{v(r)}{v(R)}\geq \frac{\int_0^r\sin_{\underline{\sk}/(\sN-1)}^{\sN-1}tdt}{\int_0^R\sin_{\underline{\sk}/(\sN-1)}^{\sN-1}tdt}.
\end{align}
If $N=1$ and $\VK\leq 0$, then
\begin{align*}
\frac{s(r)}{s(R)}\geq 1,\ \ \ \ \frac{v(r)}{v(R)}\geq \frac{r}{R}.
\end{align*}
\end{theorem}
\begin{proof}
\textbf{1.} Let us fix a point $x_0\in X$ such that $\m_{\sX}(\left\{x_0\right\})=0$, and let $R>0$ be sufficiently small such that $\VK|_{B_{2R}(x_0)}\geq \underline{\VK}$ for some $\underline{\VK}\in\mathbb{R}$. 
Let $r\in (0,R)$ and put $t=r/R$. We choose $\epsilon>0$ and $\delta>0$ and define $A_0=B_{\epsilon}(x_0)$ and $A_1=\bar{B}_{R+\delta R}(x_0)\backslash B_R(x_0)$. By triangle inequality one easily verifies that
\begin{align*}
A_t\subset \bar{B}_{r+\delta r+\epsilon r/R}(x_0)\backslash B_{r-\epsilon r/R}(x_0)\subset B_{2R}(x_0).
\end{align*}
Hence, if we consider measures $\mu_i=\m_{\sX}(A_i)^{-1}\m_{\sX}|_{A_i}$ for $i=0,1$ the curvature-dimension condition, $\m_{\sX}(\left\{x_0\right\})=\emptyset$, local finitness of the reference measure and the monotonicity of the distortion coefficients imply that
\begin{align*}
\m_{\sX}(\bar{B}_{(1+\delta)r}(x_0)\backslash B_r(x_0))^{1/N}\geq \tau_{\underline{\VK},N}^{\sscr{(r/R)}}((1\pm\delta)R)\m_{\sX}(\bar{B}_{(1+\delta)R(x_0)}\backslash B_R(x_0))^{1/N}.
\end{align*}
Since $\m_{\sX}$ is locally finite, we can assume that the right hand side is finite.
\smallskip\\
\textbf{2.} Now, we can follow precisely the proof of Theorem 2.3 in \cite{stugeo2} to obtain that $\m_{\sX}(\partial B_r(x_0))=0$ for $r\in (0,R)$, $\m_{\sX}(\left\{x\right\})=0$ for $x\in B_R(x_0)\backslash\left\{x_0\right\}$
and (\ref{one}) and (\ref{two}) for 
$r\in (0,R)$ and $R>0$ as chosen like in the first step. 
If $\m_{\sX}(\left\{x_0\right\})\neq 0$, we can choose a point $x$ close to $x_0$ such that $\m_{\sX}(\left\{x\right\})=0$ and $B_R(x)\subset B_{2R}(x_0)$.
This is implied by the local finiteness of $\m_{\sX}$ and the existence
of $\epsilon$-geodesics. If there is no such point $x$ then necessarily $\supp\m_{\sX}=\left\{x_0\right\}$.
We repeat the previous steps for $x$ instead of $x_0$ and obtain that $\m_{\sX}(\left\{x_0\right\})=0$ unless $\supp\m_{\sX}=\left\{x_0\right\}$. 
\smallskip\\
\textbf{3.} Hence, for any $x_0\in X$ there is $R>0$ (sufficiently small) such that $\de_{\sX}$ and $\m_{\sX}$ restricted to $\bar{B}_R(x_0)$ satisfy a 
doubling property provided the radius of balls is sufficiently small, and therefore $\bar{B}_r(x_0)$ is compact 
for $r\in (0,R)$. In particular, $X$ is locally compact. Then, since $(X,\de_{\sX})$ is also a complete length space, the generalized Hopf-Rinow theorem (for instance, see Theorem 2.5.28 in \cite{bbi}) implies $(X,\de_{\sX})$ is a proper metric space. 
Therefore, any closed ball $\bar{B}_R(x_0)$ is compact, and we can repeat the previous step for any $0<r<R$. In particular, it follows that (\ref{one}) and (\ref{two}) hold, and any 
bounded set has finite measure.
\end{proof}
\begin{corollary}[Doubling]
For each metric measure space $(X,\de_{\sX},\m_{\sX})$ satisfying the condition $CD(\VK,N)$ for an admissible $\VK$ and $N\geq 1$ the doubling property holds on each bounded set $X'\subset X$, and in the case $\VK\geq 0$ the doubling
constant is $\leq 2^N$, and otherwise it can be estimated in terms of $\underline{\VK}$, $N$ and $L$ as follows 
\begin{align*}
C\leq 2^N\frc_{\sk/(\sN-1)}^{\sN-1}L
\end{align*}
where $L$ is the diameter of the bounded set $X'$, and $\underline{\VK}=\min_{X'}\VK$.
\end{corollary}
\begin{proof}
The result follows from the previous theorem (see also \cite{stugeo2}).
\end{proof}
\begin{corollary}[Hausdorff dimension]
For each metric measure space $(X,\de_{\sX},\m_{\sX})$ satisfying a curvature-dimension condition $CD(\VK,N)$ for some admissible $\VK$ and $N\geq 1$, the Hausdorff dimension is $\leq N$.
\end{corollary}
\begin{definition}
Let $(X,\de_{\sX},\m_{\sX})$ be any metric measure space, let $N\geq 1$ and let $\VK:X\rightarrow \mathbb{R}$ be admissible. 
We define the \textit{effective diameter} of $(X,\de_{\sX},\m_{\sX})$ with respect to $\VK$ and $N$ as 
\begin{align*}
\pi_{\sk/(N-1)}=\sup\left\{\de_W(\mu_0,\mu_1):\exists\Pi\in \DyCo(\mu_0,\mu_1) \mbox{ s.t.}\int\tau_{\sk^{+/-},\sN}^{(t)}(|\dot{\gamma}|)d\Pi(\gamma)<\infty\right\}.
\end{align*}
By definition, we have $\pi_{\sk/(\sN-1)}\leq \diam_{\sX}$.
\end{definition}
\begin{proposition}
Let $(X,\de_{\sX},\m_{\sX})$ satisfy $CD(\VK,N)$ for an admissible function $\VK$ and $N\geq 1$. Then $\pi_{\sk/(\sN-1)}=\diam_{\sX}$.
\end{proposition}
\begin{proof}
Assume $\pi_{\sk/(\sN-1)}<\diam_{\sX}$. Then, there are points $x,y\in X$ such that $\de_{\sX}(x,y)>c+\pi_{\sk/(\sN-1)}$ for some $c>0$. Therefore, 
we can consider $\epsilon$-balls $B_{\epsilon}(x)=A_0$ and $B_{\epsilon}(y)=A_1$ such that $$\de_{\sX}(A_0,A_1):=\inf_{x_0\in A_0,x_1\in A_1}\de_{\sX}(x_0,x_1)>\pi_{\sk/(\sN-1)}.$$ 
If we define $\mu_{0/1}=\m_{\sX}(A_{0/1})^{-1}\m_{\sX}|_{A_{0/1}}$, we
see that $\de_W(\mu_0,\mu_1)>\pi_{\sk/(\sN-1)}$. Hence for each dynamcial optimal transference plan $\Pi\in \DyCo(\mu_0,\mu_1)$
\begin{align*}
\infty\leq \int \tau_{\sk^-,\sN}^{(1-t)}(|\dot{\gamma}|)d\Pi(\gamma)\m_{\sX}(A_0)^{\frac{1}{N}}+ \int \tau_{\sk^+,\sN}^{(t)}(|\dot{\gamma}|)d\Pi(\gamma)\m_{\sX}(A_1)^{\frac{1}{N}}.
\end{align*}
But by the curvature-dimension condition the right hand side is smaller thatn
\begin{align*}
-S_{\sN}(\mu_t|\m_{\sX})\leq \m_{\sX}(A_t)^{\frac{1}{N}}\leq \m_{\sX}(B_R(o))^{\frac{1}{N}}
\end{align*}
for some $o\in X$ and $R>0$ sufficiently large such that $A_t\subset B_R(o)$. $A_t$ is the set of all $t$-midpoints between $A_0$ and $A_1$.
But by the Bishop-Gromov comparison tells us that balls have always finite measure.
\end{proof}
\begin{definition}
Fix a point $x\in X$. Since $\partial B_r(x)$ is compact, we can consider $\min_{\partial B_r(x)}\VK=\VK_x(r)$ for $r<R_x$ 
where $R_x=\sup\left\{r>0:\partial B_r(x)\neq \emptyset\right\}$. 
Let $\underline{\VK}_x$ be the lower semi-continuous envelope of $\VK_x$. It is clear that $\underline{\VK}_x\leq \VK$ and $\underline{\VK}_x$ induces a
lower semi-continuous function on $X$ - also denoted by $\underline{\VK}_x$ - via
\begin{align*}
y\mapsto \underline{\VK}_x(y):=\underline{\VK}_x(\de_{\sX}(x,y)).
\end{align*}
\end{definition}
\begin{theorem}\label{bg}
Let $X$ be a metric measure space satisfying $CD(\VK,N)$. If $N>1$ then for each $x_0\in X$, for all $0<r<R$ such that $R\leq \pi_{\underline{\sk}_x/(N-1)}$, we have
\begin{align}
\frac{s(r)}{s(R)}\geq \frac{\sin_{\underline{\sk}_x/(\sN-1)}^{\sN-1}r}{\sin_{\underline{\sk}_x/(\sN-1)}^{\sN-1}R}\ \ \&\ \ 
\frac{v(r)}{v(R)}\geq \frac{\int_0^r\sin_{\underline{\sk}_x/(\sN-1)}^{\sN-1}tdt}{\int_0^R\sin_{\underline{\sk}_x/(\sN-1)}^{\sN-1}tdt}.
\end{align}
\end{theorem}
\begin{proof}
First
\begin{align*}
\inf_{\de_{\sX}(x,z)}\left\{\underline{\VK}_x(\de_{\sX}(x,z))+n|\de_{\sX}(x,z)-\de_{\sX}(x,y)|\right\}=\underline{\VK}_{x,n}(\de_{\sX}(x,y))
\end{align*}
and since $\underline{\VK}_{x,n}(r)\uparrow \underline{\VK}_x(r)$ we have $\underline{\VK}_{x,n}(\de_{\sX}(x,y))=:\underline{\VK}_{x,n}'(y)\uparrow \underline{\VK}_{x}(y)$.
By monotonicity with respect to the curvature function $X$ satisfies $CD(\underline{\VK}'_{x,n},N)$. Hence, if we consider $0<r<R<R_x$, and $A_i$ with $\mu_i$ for $i=0,1$ as
in Theorem \ref{useful} (replace $x_0$ by $x$), we obtain
\begin{align*}
&\m_{\sX}({\textstyle\bar{B}_{(1+\delta+\epsilon)r}(x)\backslash B_{(1-\epsilon)r}(x)})^{\scriptscriptstyle{\frac{1}{N}}}\! \geq\! \int\!\tau_{\underline{\VK}_{x,n,\gamma},N}^{\sscr{(r/R)}}(|\dot{\gamma}|)d\Pi_{n,\epsilon,\delta}(\gamma)\m_{\sX}({\textstyle \bar{B}_{(1+\delta)R}(x)\backslash B_R(x)})^{\scriptscriptstyle{\frac{1}{N}}}
\\
&\hspace{1cm} 
+ \int\tau_{\underline{\VK}_{x,n,\gamma},N}^{\sscr{(1-r/R)}}(|\dot{\gamma}|)d\Pi_{n,\epsilon,\delta}(\gamma)\m_{\sX}(\bar{B}_{\epsilon}(x))^{\scriptscriptstyle{\frac{1}{N}}}
\end{align*}
where $\Pi_{n,\epsilon,\delta}$ is an optimal dynamical plan between $\mu_0$ and $\mu_1$. Since the left hand side is finite, 
the right hand side is uniformily bounded and 
the distortion coefficients are finite almost everywhere.
If $\epsilon\rightarrow 0$, compactness of closed balls implies that we can find 
a subsequence of $\Pi_{n,\epsilon,\delta}$ that converges to $\Pi_{n,\delta}$ for $n\rightarrow \infty$ and with $(e_0)_*\Pi_{n,\delta}=\delta_x$. 
The previous inequality becomes
\begin{align*}
\m_{\sX}(\bar{B}_{(1+\delta)r}(x)\backslash B_{r}(x))\geq \left(\int\tau_{\underline{\VK}_{x,n,\gamma},N}^{\sscr{(r/R)}}(|\dot{\gamma}|)d\Pi_{n,\epsilon,\delta}(\gamma)\right)^N\m_{\sX}(\bar{B}_{(1+\delta)R}(x)\backslash B_R(x))
\end{align*}
We remark that $\gamma\mapsto\tau_{\underline{\VK}_{x,n,\gamma},N}^{\sscr{(r/R)}}(|\dot{\gamma}|)$ is bounded and continuous for geodesics $\gamma$ in a sufficiently large ball.
Similar, if $\delta$ goes to $0$, we can take another sub-sequence of $\Pi_{n,\delta}$ that converges to $\Pi_n$. If we devide both side by $\delta r$ and take $\delta \rightarrow 0$, 
the previous inequality becomes 
\begin{align*}
s_x(r)\geq \left(\int\sigma_{\underline{\VK}_{x,n,\gamma}/(N-1)}^{\sscr{(r/R)}}(|\dot{\gamma}|)d\Pi_{n}(\gamma)\right)^Ns_x(R).
\end{align*}
$(e_0)_*\Pi_n=\delta_x$ and $(e_1)_*\Pi_n$ is a probability measure with $(e_1)_*\Pi_n(\partial B_R(x))=1$. Hence $\Pi_n$ is supported on geodesics with 
$\gamma(0)=x$ and $|\dot{\gamma}|=R$, and by definition of $\underline{\VK}_{x,n}'$ we have that $\underline{\VK}_{x,n}'\circ\bar{\gamma}=\underline{\VK}_{x,n}'(\cdot R)$ is independent of $\gamma$. Therefore
\begin{align*}
\frac{s_x(r)}{s_x(R)}\geq \sigma_{\underline{\VK}_{x,n}/(N-1)}^{\sscr{(r/R)}}(R)^{N-1}.
\end{align*}
Now, take $n\rightarrow \infty$. Since $\underline{\VK}_{x,n}\uparrow \underline{\VK}_x$, 
one can check as in Lemma \ref{gr} that - after choosing another subsequence - $\frs_{\underline{\VK}_{x,n}}\downarrow \frs_{\underline{\VK}_x}$. 
This is the first claim. The second one follows as in Theorem \ref{useful}.
\end{proof}

\begin{theorem}\label{stheorem}
Let $(X,\de_{\sX},\m_{\sX})$ be a metric measure space satisfying $CD(\VK,N)$ for $N>1$. Assume there is point $x_0\in X$, a constant $c>\frac{N-1}{4}$
and $R>0$ such that
\begin{align*}
\VK(x)\geq c\de_{\sX}(p,x)^{-2}\ \ \mbox{ for all }x\in X \mbox{ with }\de_{\sX}(p,x)>R.
\end{align*}
Then $X$ is compact.
\end{theorem}
\begin{proof}
Choose $\alpha>0$ such that 
\begin{align*}
{\textstyle \frac{1}{4}}(N-1)<\left(\textstyle{\frac{1}{4}}+\alpha^2\right)(N-1)<c
\end{align*}
Assume $(X,\de_{\sX},\m_{\sX})$ is not compact. 
Then we can find a point $q\in X$ such that $\de_{\sX}(p,q)>(R+\delta)e^{\frac{\pi}{\alpha}}$ for some $0<\delta<R$.
We choose $\delta>0$ and $\epsilon>0$ such that $2\epsilon(2-e^{-\frac{\pi}{\alpha}})<\delta$ 
and 
\begin{align*}
\min_{x\in B_{\epsilon}(p),y\in B_{\epsilon}(q)}\de_{\sX}(x,y)=:\de_{\sX}(\bar{B}_{\epsilon}(q),\bar{B}_{\epsilon}(p))>(R+\delta)e^{\frac{\pi}{\alpha}}.
\end{align*}
We set $\bar{B}_{\epsilon}(q)=:A_0$ and $\bar{B}_{\epsilon}(p)=:A_1$ and define probability measures $$\mu_{i}={\m_{\sX}(A_i)^{-1}}\mu_{\sX}|_{A_i}$$
where $i=0,1$.
Let $q'\in \bar{B}_{\epsilon}(q)$ and $p'\in \bar{B}_{\epsilon}(p)$. We consider a geodesic $\gamma:[0,1]\rightarrow X$ between $q'$ and $p'$ and 
estimate the curvature along $\gamma$ as follows. Let $\bar{\gamma}$ be the unit speed reparametrization of $\gamma$. For $0<t<[\de_{\sX}({q'},{p'})+2\epsilon](1-e^{-\frac{\pi}{\alpha}})$ we have
\begin{align*}
\de_{\sX}(p,\bar{\gamma}(t))&\geq \de_{\sX}(p',\gamma(t))-\de_{\sX}(p,p')\geq[\de_{\sX}(p',q')-t]-\epsilon\\
&>\de_{\sX}(q',p')e^{-\frac{\pi}{\alpha}}-2\epsilon(1-e^{-\frac{\pi}{\alpha}})-\epsilon\\
&\geq(R+\delta)-\epsilon(2-e^{-\frac{\pi}{\alpha}})>R
\end{align*}
Therefore
\begin{align*}
\VK(\bar{\gamma}(t))&\geq \frac{c}{\de_{\sX}(p,\bar{\gamma}(t))^{2}}\geq\frac{c}{(\de_{\sX}({p},{p'})+\de_{\sX}(p',\bar{\gamma}(t))^{2}}\\
&\geq (\alpha^2+{\frac{1}{4}})(N-1)\frac{1}{(\epsilon+\de_{\sX}(q',p')-t)^2}=:\kappa(t)(N-1)
\end{align*}
We obtain a lower estimate for the modified distortion coefficient along $\gamma$. The generalized $\sin$-function $\frs_{\VK\circ\bar{\gamma}/(N-1)}$ is bounded from below by $\frs_{\kappa}$ which is given explicetly by
\begin{align*}
\frs_{\kappa}(t)=C\sqrt{\epsilon+\de_{\sX}({p'},{q'})-t)}\sin\left[\alpha \log\left({\textstyle\frac{\epsilon+\de_{\sX}(q',p')-t}{\left(\epsilon+\de_{\sX}(q',p')\right)e^{-\pi/\alpha}}}\right)\right].
\end{align*}
where $C$ is a normalization constant. We see that the second zero of $\frs_{\kappa}$ appears at 
$$(\epsilon+\de_{\sX}({q'},{p'}))(1-e^{-\frac{\pi}{\alpha}})<\de_{\sX}({q'},{p'})-R+\epsilon(1-e^{-\frac{\pi}{\alpha}})<\de_{\sX}({q'},{p'}).$$
Therefore, the second zero of $\frs_{\VK\circ\bar{\gamma}}$ appears strictly before $t=\de_{\sX}({q},{p})$.
Consequently $$\sigma_{\VK\circ\gamma,N-1}^{(t)}(\theta)\geq \sigma_{\kappa}^{(t)}(\theta)=\infty.$$
We conclude that 
\begin{align*}
\m_{\sX}(A_t)^{\frac{1}{N}}&\geq {\int\tau_{\VK^{\sscr{-}}_{\gamma},N'}^{\sscr{(1-t)}}(|\dot{\gamma}|)d\Pi(\gamma)}\m_{\sX}\left(A_0\right)^{\frac{1}{N'}}+{\int\tau_{\VK^{\sscr{+}}_{\gamma},N'}^{\sscr{(t)}}(|\dot{\gamma}|)d\Pi(\gamma)}\m_{\sX}\left(A_1\right)^{\frac{1}{N'}}{=\infty}.
\end{align*}
$A_t$ is again the set of all $t$-midpoints between $A_0$ and $A_1$, and $\Pi$ is an optimal dynamical transference for $\mu_0$ and $\mu_1$. As in the previous Proposition this yields a contradiction.
Hence, $X$ is compact.
\end{proof}
\begin{example}
The previous theorem is sharp in the sense that one can not improve the result by replacing the lower bound $\frac{1}{4}(N-1)$ for $c$ by a smaller lower bound. For instance,
consider 
\begin{align*}
([0,\infty),|\cdot|_2,\left(\sqrt{r}\right)^{\sN-1}dr).
\end{align*}
Using Theorem \ref{smoothcase} and Proposition \ref{prop73}
one can check that it satifies the curvature-dimension $CD(\VK,N)$ for 
\begin{align*}
\VK(r)= \frac{1}{4}(N-1)r^{-2}
\end{align*}
$\VK$ satisfies the assumption of the theorem for $c=\frac{1}{4}(N-1)$ and any $p\in [0,\infty)$ since
$
\VK(r)\sim\frac{1}{4}(N-1)|r-p|_2^{-2} 
$
for $r>0$ sufficiently large
but one cannot find a point $p\in[0,\infty)$, $c>\frac{1}{4}(N-1)$ and $R>0$ such that $\VK(r)r^2\geq c$ for $r>0$ with $|r-p|_2\geq R$. 
A Riemannian manifold of geometric dimension $N$ satisfying this property can be constructed via warped products.
\end{example}

%
\section{Stability}
\paragraph{\textbf{Measured Gromov-Hausdorff convergence.}}
A rectifiable curve $\gamma:[0,1]\rightarrow X$ with constant speed parametrization 
is called $\epsilon$-geodesic if $\mbox{L}(\gamma)-\epsilon<\de_{\sX}(\gamma(0),\gamma(1))$.
The family of all $\epsilon$-geodesics is denoted with $\mathcal{G}^{\epsilon}(X)$, and it is equipped with the topology that comes from
$\de_{\infty}(\gamma,\tilde{\gamma})=\sup_{t}\de_{\sX}(\gamma(t),\tilde{\gamma}(t))$. Measurability is understood in the sense of this topology.
Obviously, we have $\mathcal{G}^{\epsilon}(X)\subset \mathcal{G}^{\eta}(X)$ if $\epsilon\leq \eta$ and $\mathcal{G}^{0}(X)=\mathcal{G}(X)$.
If $X$ is compact, then $\mathcal{G}^{\epsilon}(X)$ is compact with respect to $\de_{\infty}$ by suitable version of the Arzela-Ascoli theorem.
\smallskip

We need an extension of the notion of dynamical transference plan on $\mathcal{G}(X)$.
The evaluation map $e_t:\gamma \mapsto \gamma(t)$ is continuous and measurable.
A probability measure $\Pi$ on $\mathcal{G}^{\epsilon}(X)$ is called \textit{dynamical transference plan} between $(e_0)_{\star}\Pi$ and $(e_1)_{\star}\Pi$.
If $\VK:X\rightarrow \mathbb{R}$ is an admissible function, we can consider $\VK_{\gamma}$ for $\gamma\in\mathcal{G}^{\epsilon}(X)$ and the corresponding 
generalized $\sin$-function and the modified distortion coefficient.
One can check that 
$\gamma\mapsto \tau_{\sk^+_{\gamma},\sN}^{\sscr{(t)}}(|\dot{\gamma}|)$
is measurable on $\mathcal{G}^{\epsilon}(X)$.
\begin{definition}
Let $(X,\de_{\sX})$ be a metric space which is separable, complete and compact. A subset $D\subset X$ is $\epsilon$-dense for $\epsilon>0$ if $B_{\epsilon}(D)=X$.
\end{definition}
\begin{definition}
Let $(X,\de_{\sX})$ and $(Y,\de_{\sY})$ be metric spaces. A map $f:X\rightarrow Y$ is called $\epsilon$-isometry from $X$ to $Y$ if $f(X)$ is $\epsilon$-dense in $Y$ and for any pair $x,y\in X$
\begin{align*}
|\de_{\sX_i}(x,y)-\de_{\sX}(f_i(x),f_i(y))|<\epsilon.
\end{align*}
We say that $X$ and $Y$ have finite Gromov-Hausdorff distance if there exist an $\epsilon$-isometry from $X$ to $Y$.
\end{definition}
\begin{definition}
A sequence $(X_i,\de_{\sX_i})_{i\in\mathbb{N}}$ of compact metric spaces converges in Gromov-Hausdorff sense to a metric space $(X,\de_{\sX})$ 
if for each $i\in\mathbb{N}$ there exist $\epsilon_i>0$ and an $\epsilon_i$-isometry $f_i:X_i\rightarrow X$ such that $\epsilon_i\rightarrow 0$ for $i\rightarrow \infty$. A sequence of metric measure spaces
$(X_i,\de_{\sX_i},\m_{\sX_i})$ converges in measured Gromov-Hausdorff sense to a metric measure space $(X,\de_{\sX},\m_{\sX})$ if the corresponding metric spaces converge in Gromov-Hausdorff sense and
\begin{align*}
(f_i)_{\star}\m_{\sX_i}\longrightarrow \m_{\sX} \ \ \mbox{ with respect to weak convergence.}
\end{align*}
\end{definition}
\begin{remark}\label{ao}
For fixed $i\in\mathbb{N}$ 
the existence of an $\epsilon_i$-isometry as in the previous definition implies the existence of a metric space $(Z,\de_{\sZ})$ and isometric embeddings
$\iota_{i},\iota:(X_i,\de_{\sX_i}),(X,\de_{\sX})\rightarrow (Z,\de_{Z})$ such that $\iota_{i}(X_i)$ and 
$\iota(X)$ are $2\epsilon_i$-close w.r.t. the Hausdorff distance (see Corollary 
7.3.28 in \cite{bbi}). More precisely, we can choose $Z$ as the disjoint union of $X_i$ and $X$, and 
$\de_{\sZ}|_{X_i^2}=\de_{\sX_i}$, $\de_{\sZ}|_{X^2}=\de_{\sX}$ and 
\begin{align*}
\de_{\sZ}(z_1,z_2)=\epsilon_i+\inf_{x\in X_i}\left[\de_{\sX}(z_1,f_i(x))+\de_{\sX_i}(x,z_2)\right]&\mbox{ if }z_1\in X,z_2\in X_i.
\end{align*}
In particular, $\de_{\sZ}(x,f_i(x))=\epsilon_i$.
Additionally, $f_i$ induces a coupling between $\m_{\sX_i}$ and $(f_i)_{\star}\m_{\sX_i}$ such that
\begin{align*}
\int_{Z^2}\de_{\sZ}(x,y)^2d\bar{q}_i(x,y)\leq\left\|\de_{\sZ}(x,y)\right\|^2_{L^{\infty}(Z,\bar{q}_i)}<\epsilon_i^2
\end{align*}
where $\bar{q}_i=(\mbox{id}_{\sX_i},f_i)_{\star}\m_{\sX_i}$. The previous estimate can be found for instance in the proof of Lemma 3.18 in \cite{stugeo1}.
In the following, if $i$ is fixed, we will always
identify $(X_i,\de_{\sX_i})$ and $(X,\de_{\sX})$ with their embeddings in $Z$, and $\m_{\sX_i}$ and $\m_{\sX}$ with their pushfowards with respect to $\iota_{\sX_i}$ and $\iota_{\sX}$ 
respectively.
Therefore, $f_i$ yields an $L^{\infty}$-coupling between $\m_{\sX_i}$ and $(f_i)_{\star}\m_{\sX_i}$ in $(Z,\de_{\sZ})$.
\end{remark}
\begin{proposition}[\cite{lottvillaniweakcurvature}]
Let $(X,\de_{\sX})$ be a compact length space. Then for all $\epsilon>0$ there exists $\delta>0$ with the following property. 
If $(Y,\de_{\sY})$ is compact length space, $f:Y\rightarrow X$ is a $\delta$-isometry and $\gamma:[0,1]\rightarrow Y$ is
a geodesic, then there exists a geodesic $\gamma':[0,1]\rightarrow X$ such that $\de_{\infty}(\gamma',f(\gamma))<\epsilon$. 
Additionally, one can choose $\gamma\mapsto \gamma'$ to be a measurable map from $\mathcal{G}(Y)$ to $\mathcal{G}(X)$.
\end{proposition}
\begin{remark}
Let $\epsilon$, $\delta$, $Y$, $f$ and $\gamma$ as in the previous proposition, and we choose $Z$ such that $X,Y$ embed into $Z$.
Then $\de_{\infty}(\gamma,\gamma')\leq \epsilon+\delta$ where $\de_{\infty}$ is w.r.t. $\de_{\sZ}$.
\end{remark}

\begin{remark}\label{lottvillani}
Let $(X_i,\de_{\sX_i})$ be a sequence that converges in Gromov-Hausdorff sense to $(X,\de_{\sX})$. By the previous proposition
for each $\epsilon>0$ there exists $i_{\epsilon}\in \mathbb{N}$ such that for $i\geq i_{\epsilon}$ and for each $\gamma\in\mathcal{G}(X_i)$, 
one can finde a constant speed curve
$\tilde{\gamma}_i:[0,1]\rightarrow X$ with endpoints
$f_i(\gamma(0))=\tilde{\gamma}_i(0)$ and $f_i(\gamma(1))=\tilde{\gamma}_i(1)$ such that $\tilde{\gamma}_i$ and 
$\gamma$ are $(\epsilon+3\epsilon_i)$-close with respect to $\de_{\infty}$ in $Z$ and 
\begin{align}\label{pr}
\mbox{L}(\tilde{\gamma}_i)\leq 
\de_{\sX}(\tilde{\gamma}_i(0),\tilde{\gamma}_i(1))+2\epsilon_i.\end{align}
$\tilde{\gamma}_i$ is given by $\Psi(f(\gamma(0)),\gamma'(0))*\gamma'*\Psi(\gamma'(1),f(\gamma(1))):[0,3]\rightarrow X$ where $\gamma'$ 
is the curve from the previous proposition.
Here, the operator $*$ denotes the catenation of curves. More precisely, we define a rectifiable curve $c:[0,3]\rightarrow X$ via
\begin{align*}
c(t)=\begin{cases}
                                                                   \Psi(f(\gamma(0)),\gamma'(0))(t)& \mbox{ if }t\in[0,1]\\
                                                                   \gamma'(t-1)& \mbox{ if }t\in[1,2]\\
                                                                   \Psi(\gamma'(1),f(\gamma(1)))(t-2)&\mbox{ if } t\in[2,3]
                                                                   \end{cases}
\end{align*}
and then we define $\Psi(f(\gamma(0)),\gamma'(0))*\gamma'*\Psi(\gamma'(1),f(\gamma(1)):[0,1]\rightarrow X$ as the constant speed reparametrization of $c$.
The map
$\Phi_i: \mathcal{G}(X_i)\rightarrow \mathcal{G}^{2\epsilon_i}(X)$ with $i\geq i_{\epsilon}$ and 
$\gamma\mapsto \tilde{\gamma}_i=:\Phi_i(\gamma)$ can be chosen measurable. In the following we will choose $i_{\epsilon}$
such that $3\epsilon_i<\epsilon$. Therefore, $\gamma$ is $2\epsilon$-close to $\Phi_i(\gamma)$ in $X$.
\end{remark}
\begin{definition}\label{\VK}
Let $(X_i,\de_{\sX_i})$ be metric measure spaces converging to a metric space $(X,\de_{\sX})$ in
Gromov-Hausdorff sense. Let $\VK_i,\VK:X_i,X\rightarrow \mathbb{R}$ be admissible functions. We say
\begin{align*}
\liminf_{i\rightarrow\infty}\VK_i\geq \VK
\end{align*}
if for each $\eta>0$ there exists $i_{\eta}\in\mathbb{N}$ such that $\VK_i(x)\geq \VK(f_i(x))-\eta$ if $i\geq i_{\eta}$ for each $x\in X_i$.
\end{definition}
\paragraph{\textbf{Stability of the curvature-dimension condition}}
\begin{theorem}\label{stabilitytheorem}
For $i\in{\mathbb{N}}$ let $(X_i,\de_{\sX_i},\m_{\sX_i})$ be metric measure spaces that satisfy $CD(\VK_i,N_i)$ respectively for admissible functions $\VK_i$ and $N_i\in[1,\infty)$. 
Assume $X_i$ converges to $(X,\de_{\sX},\m_{\sX})$ 
in measured 
Gromov-Hausdorff sense, and consider an admissible function $\VK:X\rightarrow \mathbb{R}$ and $N\in[1,\infty)$ such that 
\begin{align*}
\liminf_{i\rightarrow\infty}\VK_i\geq \VK\ \ \& \ \ \limsup_{i\rightarrow \infty}N_i\leq N\ \ \& \ \ \diam_{\sX_i}\leq L
\end{align*}
Then $(X,\de_{\sX},\m_{\sX})$ satisfies $CD(\VK,N)$.
\end{theorem}
\begin{lemma}\label{therethere}
Let $\VK:X\rightarrow \mathbb{R}$ be admissible and $N>1$. For dynamical couplings 
$(\Pi_n)_{n\in\mathbb{N}}$ supported on $\mathcal{G}^{\eta}(X)$ for some $\eta>0$ with the same marginals $\mu_0, \mu_1\in \mathcal{P}(X,\m_{\sX})$ which converge to a dynamical coupling $\Pi_{\infty}$, it follows
\begin{align*}
\limsup_{n\rightarrow\infty}T^{\sscr{(t)}}_{\VK,N}(\Pi_n|\m_{\sX})\leq T^{\sscr{(t)}}_{\VK,N}(\Pi_{\infty}|\m_{\sX}).
\end{align*}
\end{lemma}
\begin{proof}First, we assume that $\VK$ is continuous.
We will show that 
\begin{align*}
\liminf_{n\rightarrow \infty}\int\tau_{\sk^+_{\gamma},\sN}^{\sscr{(t)}}(|\dot{\gamma}|)\varrho_0(\gamma_0)^{-\frac{1}{N}}\Pi_n(d\gamma)\geq \int\tau_{\sk^+_{\gamma},\sN}^{\sscr{(t)}}(|\dot{\gamma}|)\varrho_0(\gamma_0)^{-\frac{1}{N}}\Pi_{\infty}(d\gamma).
\end{align*}
Let $\Pi_{n,x_0}(d\gamma)$ be a disintegration of $\Pi_n$ with respect to $\mu_0$ for $n\in\mathbb{N}\cup\left\{\infty\right\}$, and let $C\in(0,\infty)$. We put
\begin{align*}
v_{0,n}^{C}(x_0):=\int\left[\tau_{\sk^+_{\gamma},\sN}^{\sscr{(t)}}(|\dot{\gamma}|)\wedge C\right] \Pi_{n,x_0}(d\gamma).
\end{align*}
where $n\in\mathbb{N}\cup \left\{\infty\right\}$. Since $C_b(X)$ is dense in $L^1(\m_{\sX})$, and since $v_{0,n}^C$ is bounded by definition, for each $\epsilon>0$ there is $\psi\in C_b(X)$ such that 
\begin{align}\label{alpha}
\int v_{0,n}^C|\varrho_0^{-\frac{1}{N}}\wedge C-\psi|d\mu_0<\epsilon \ \ \mbox{ for all } \ \ n\in\mathbb{N}\cup \left\{\infty\right\}
\end{align}
if $C<\infty$.
Weak convergence of $\Pi_n\rightarrow \Pi_{\infty}$ on $\mathcal{G}^{\eta}(X)$ implies that one can find $n_{\epsilon}$ such that for each $n\geq n_{\epsilon}$, one has
\begin{align}\label{beta}
  \int v_{0,\infty}^C\psi d\mu_0\leq \int v_{0,n}^C\psi d\mu_0+\epsilon
\end{align}
Putting together (\ref{alpha}) and (\ref{beta}) one gets
\begin{align*}
\int v_{0,\infty}^C[\varrho_0^{-\frac{1}{N}}\wedge C]d\mu_0\leq \int v_{0,n}^C[\varrho_0^{-\frac{1}{N}}\wedge C]d\mu_0+3\epsilon\leq \int v_{0,n}^{\infty}\varrho_0^{-\frac{1}{N}}d\mu_0+3\epsilon.
\end{align*}
It follows that for each $C>0$
\begin{align}\label{gruen}
\int v_{0,\infty}^C\varrho_0^{-\frac{1}{N}}\wedge C d\mu_0\leq \liminf_{n\rightarrow \infty}\int v_{0,n}^{\infty}\varrho_0^{-\frac{1}{N}}d\mu_0.
\end{align}
Finally, let $C\rightarrow \infty$
\begin{align*}
\int \tau_{\sk^+_{\gamma},\sN}^{\sscr{(t)}}(|\dot{\gamma}|)\varrho_0(\gamma_0)^{-\frac{1}{N}}\Pi(d\gamma)\leq \liminf_{n\rightarrow \infty}\int v_{0,n}^{\infty}\varrho_0^{-\frac{1}{N}}d\mu_0.
\end{align*}
The same statement holds with $\varrho_0$ replaced by $\varrho_1$ and $\tau^{\sscr{(t)}}_{\sk^+_{\gamma},\sN}$ replaced by $\tau_{\sk^-_{\gamma},\sN}^{\sscr{(1-t)}}$.
\smallskip\\
Now, let $\VK$ be lower semi-continuous, and let $\VK_i$ be a sequence of continuous functions that converge pointwise monotone from below to $\VK$.
By monotonicity of the distortion coefficients we observe that $$\tau_{\VK_{i,\gamma}^+,\sN}^{(t)}(|\dot{\gamma}|)\uparrow \tau_{\VK_{\gamma}^+,\sN}^{(t)}(|\dot{\gamma}|)$$
for any $\gamma\in\mathcal{G}^{\epsilon}$. Therefore, 
\begin{align*}
v_{0,\infty,i}^C\uparrow v_{0,\infty}^C\mbox{ and }v_{0,n,i}^{\infty}\uparrow v_{0,n}^{\infty} \mbox{ if }i\rightarrow \infty.
\end{align*}
In particular, for $\epsilon>0$ we can choose $i_{\epsilon}\in\mathbb{N}$ such that for $i\geq i_{\epsilon}$
\begin{align*}
\int \left[v_{0,\infty,i}^C-v_{0,\infty}^C\right]\varrho_0^{-\frac{1}{N}}\wedge C d\mu_0<\epsilon
\end{align*}
Hence, together with (\ref{gruen}) it follows that 
\begin{align*}
\int v_{0,\infty}^C\varrho_0^{-\frac{1}{N}}\wedge C d\mu_0-\epsilon\leq \liminf_{n\rightarrow \infty}\int v_{0,n}^{\infty}\varrho_0^{-\frac{1}{N}}d\mu_0
\end{align*}
and finally we let $C\rightarrow \infty$ and $\epsilon\rightarrow 0$, and the result follows as before.
\end{proof}

\begin{proofoftheorem}
\textbf{1.} 
First, let us assume that $\VK$ is continuous. Gromov-Hausdorff convergence and $\diam_{\sX_i}\leq L$ 
yields that $(X,\de_{\sX})$ is a compact geodesic space that satisfies volume doubling and $\diam_{\sX}\leq L$. Let $(Z,\de_{\sZ})$ be 
the metric space that was introduced in Remark \ref{ao}.
\medskip
\\
Since $\liminf \VK_i\geq \VK$, for each $\eta>0$ there exist $i_{\eta}$ such that $\VK_i(x)\geq \VK(f_i(x))-\eta/2$ for any $i\geq i_{\eta}$.
Since $\VK$ is continuous, there is $\hat{\epsilon}>0$ such that $\VK(x)\geq \VK(y)-\eta/2$ if $\de_{\sX}(x,y)\leq 4\epsilon$ and $\epsilon\in (0,\hat{\epsilon}]$.
We can choose $i_{\epsilon}\geq i_{\eta}$ as in Remark \ref{lottvillani} such that $3\epsilon_i<\epsilon$ and
$
\Phi_i(\mathcal{G}(X_i))\subset \mathcal{G}^{2\epsilon}(X)
$
for $i\geq i_{\epsilon}$. It follows that $\VK_i\circ\gamma (t)\geq \VK\circ \Phi_i(\gamma)(t)-\eta$ for $i\geq i_{\epsilon}$ 
since $\gamma$ and $\Phi_i(\gamma)$ are $2\epsilon$-close.
\medskip
\\
We can assume that $\m_{\sX_i}$ are probability measures. Since $(f_i)_{\star}\m_{\sX_i}\rightarrow \m_{\sX}$ weakly, 
the $L^2$-Wasserstein distance in $X\subset Z$ goes to zero. Hence, there exists $i_0\geq i_{\epsilon}$ 
such that $\de_{W}((f_i)_{\star}\m_{\sX_i},\m_{\sX})^2<\epsilon^3/2^4$ for $i\geq i_{0}$.
In the following we consider $\eta$ and choose a $\epsilon$ and $i\geq i_{0}$ as before. 
%
\medskip\\
{\textbf 1$\frac{1}{2}$.}
Let $\hat{q}_i$ be an optimal coupling between $(f_i)_{\star}\m_{\sX_i}$ and $\m_{\sX}$ and let $\bar{q}_i$ be the coupling between $\m_{\sX_i}$ and $(f_i)_{\star}\m_{\sX_i}$ that was introduced in Remark \ref{ao}.
By gluing $\hat{q}_i$ and $\bar{q}_i$ one obtains a coupling ${q}_i$ whose total cost is less than $\epsilon_i+\epsilon^3/2\leq \epsilon$ if $i>i_{\epsilon}$. 
It provides an upper bound for the $L^2$-Wasserstein distance between $\m_{\sX_i}$ and $\m_{\sX}$ in $Z$. 
Following \cite{stugeo1} one can define a map ${Q}_i':\mathcal{P}_2(\m_{\sX})\rightarrow \mathcal{P}_2(\m_{\sX_i})$ with
\begin{align}\label{someestimate}
S_N(Q_i'(\mu)|\m_{\sX_i})\leq S_N(\mu|\m_{\sX})\ \ \& \ \ \de_{W}^2(\mu,Q_i'(\mu))<\delta(i)
\end{align}
where $\de_{W}$ denotes the Wassertstein distance in $(Z,\de_{\sZ})$ and $\delta(i)\rightarrow 0$ for $i\rightarrow \infty$.
In \cite{stugeo1} $Q_i'$ is constructed explicitely by disintegration of an optimal coupling with respect to $\m_{\sX_i}$. 
But one can see that for the estimates (\ref{someestimate}) the coupling ${q}$ is already sufficient.
More precisely, we set 
$
\mu_{j,i}=Q'_i(\mu_j)=\varrho_{j,i}d\m_{\sX_i}
$
where 
\begin{align*}
\varrho_{j,i}(y)=\int_X\varrho_{j}(x)Q'_i(y,dx)=\int_X{\int_X\varrho_{j}(z)\hat{Q}_i'(z,dx)}\bar{Q}'_i(y,dz)
\end{align*}
and $j=0,1$. $\hat{Q}_i'$ and $\bar{Q}_i'$ are disintegrations of $\hat{q}$ and $\bar{q}$ with respect to $(f_i)_{\star}\m_{\sX_i}$ and $\m_{\sX_i}$ respectively. 
In particular, $\bar{Q}_i'(y,dz)=\delta_{f_i(x)}(dz)$.
Similar, we can define $Q_i:\mathcal{P}_2(\m_{\sX_{i}})\rightarrow \mathcal{P}_2(\m_{\sX})$ by
$
\mu^{i}=Q^i(\mu)=\varrho^{i}d\m_{\sX_i}
$
where 
\begin{align*}
\varrho^{i}(x)=\int_X\varrho(y)Q^i(x,dy)=\int_X{\int_X\varrho(y)\bar{Q}^i(z,dy)}\hat{Q}^i(x,dz)
\end{align*}
where $\hat{Q}^i$ and $\bar{Q}^i$ are disintegrations of $\hat{q}_i$ and $\bar{q}_i$ with respect to $f_{\star}\m_{\sX_i}$ and $\m_{\sX}$ respectively.
Again we have
\begin{align}\label{corres}
S_N(Q^i(\mu)|\m_{\sX})\leq S_N(\mu|\m_{\sX_i})\ \ \& \ \ \de_{W}^2(\mu,Q^i(\mu))<\delta(i)
\end{align}
In the next step we will transport probability measure from $X$ to $X_i$ via $Q$, and we want to emphasize that this transport consists of two parts, corresponding to $\hat{q}$ and $\bar{q}$ respectively.
\smallskip\\
\textbf{2.} 
Pick measures $\mu_0=\varrho_0d\m_{\sX}$ and $\mu_1=\varrho_1d\m_{\sX}$ in $\mathcal{P}_2(X_i,\m_{\sX_i})$ with bounded densities.  
Due to the curvature-dimension condition on $X_i$, there exists a geodesic $\mu_{t,i}$ and a dynamical optimal transport plan $\Pi_i$ such that 
\begin{align*}
S_{N'}(\mu_{t,i}|\m_{\sX_i})\leq T_{\sk_i,\sN'}^{(t)}(\Pi_i|\m_{\sX_i}).
\end{align*}
By (\ref{corres}) we know that $S_{N'}$ decreases and $Q^i(\mu_{t,i})$ is a $\delta(i)$-geodesic in $\mathcal{P}_2(X,\m_{\sX})$ where $\delta(i)\rightarrow 0$ if $i\rightarrow \infty$.
By compactness of space $Q^i(\mu_{t,i})$ converges to a geodesic $\mu_t$ between $\mu_0$ and $\mu_1$.
In the following we always write $N$ instead of $N'$. 
\medskip\\
We will generalize the map that was introduced in Remark \ref{lottvillani}. 
We pick a geodesic $\gamma\in X_i$ and we consider the map $\Phi_i$ and the $2\epsilon$-geodesic $\Phi_i(\gamma)$. 
Since $X$ is a compact geodesic space, one can choose a measurable map $\Psi:X^2\rightarrow \mathcal{G}(X)$ such that $\Psi(x,y)$ 
is a geodesic between $x$ and $y$. For instance, this follows from a measurable selection theorem. Now we define a Markov kernel 
$
\mathcal{Q}$ on $\mathcal{G}(X_i)\times\mathcal{P}(\mathcal{G}^{4\epsilon+\diam_{\sX}}(X))$ as follows. 
Consider the map $\Xi_i:\mathcal{G}(X_i)\times X^2\rightarrow \mathcal{G}^{4\epsilon+\diam_{\sX}}(X)$ that is defined as
$$(\gamma,x_0,x_1)\mapsto \Psi(x_0,\Phi_i(\gamma)(0))*\Phi(\gamma)* \Psi(\Phi_i(\gamma)(1),x_1).$$ 
and consider ${Q}_i'(\cdot,dx)$, $\varrho_j$ and $\varrho_{j,i}$. 
Here, the operator $*$ is like in Remark \ref{lottvillani}. 
It is clear from the construction that $\Xi_i$ maps to $\mathcal{G}^{4\epsilon+\diam_{\sX}}(X)$ and $\Xi_i$ is measurable. 
We also set $\Xi_{i,\gamma}(\cdot):=\Xi_i(\gamma,\cdot)$.
\smallskip\\
Then we define 
$\mathcal{Q}(\gamma,d\sigma)=(\Xi_{i,\gamma})_{\star}P_{\gamma_0,\gamma_1}(d(x_0,x_1))$ where $${P_{\gamma_0,\gamma_1}(d(x_0,x_1)):=\left[\frac{\varrho_j(x)}{\varrho_{j,i}(\gamma(0))}{Q}_i(\gamma(0),dx_0)
\otimes \frac{\varrho_{j}(x)}{\varrho_{j,i}(\gamma(1))}{Q}_i(\gamma(1),dx_1)\right]}.$$
$\mathcal{Q}$ is a Markov kernel.
We define a dynamical transference plan $\hat{\Pi}_i$ on $\mathcal{G}^{4\epsilon+\diam_{\sX}}(X)$ via
$$
\int_{\mathcal{G}(X_i)}\mathcal{Q}(\gamma,d\sigma)\Pi_i(d\gamma)=\hat{\Pi}_i(d\sigma). \ \ \mbox{Set $(e_0,e_1)_{\star}\hat{\Pi}_i=\hat{\pi}_i$.}
$$
If $f:X^2\rightarrow \mathbb{R}$ is continuous and bounded on $X^2$, then one can compute that
\begin{align*}
&\int_{X^2}f(x_0,x_1)\hat{\pi}_i(dx_0,dx_1)=\int\int f(e_0(\sigma),e_1(\sigma))\mathcal{Q}(\gamma,d\sigma)\Pi_i(d\gamma)\\
%
%
&=\int_{X_i^2}\int_{X^2} f(x_0,x_1)\frac{\varrho_0(x_0)\varrho_1(x_1)}{\varrho_{1,i}(y_1)}Q_i'(y_1,dx_1)\pi_i(dy_0,dy_1)Q^i(x_0,dy_0)\m_{\sX}(dx_0)
\end{align*}
Since the equality holds for any $f$, we obtain an explicite formula for $\hat{\pi}_i$.
If one chooses $f(x_0,x_1)=f_0(x_0)$ or $f(x_0,x_1)=f_1(x_1)$, one can see that that the first and the final marginal of $\hat{\Pi}_i$ are $\mu_0$ and $\mu_1$ respectively.
Let $\hat{\Pi}_{i,x_0,x_1}(d\sigma)$ be a disintegration of $\hat{\Pi}_i$ with respect to $\hat{\pi}_i$.
Let $C>0$ be a constant. For $\gamma\in\mathcal{G}(X_i)$ we define
\begin{align*}
\tau_{\sk_{i,\gamma}^{-/+},\sN}^{\sscr{(1-t)/(t)}}(|\dot{\gamma}|)=b^{-/+}(\gamma)\in [0,\infty]
\end{align*}
and for $\sigma\in\mathcal{G}^{4\epsilon+\diam_{\sX}}(X)$ we define
\begin{align*}
\sigma\in \mathcal{G}^{4\epsilon+\diam_{\sX}}\mapsto a^{-/+}(\sigma)
:=\tau_{\sk_{\sigma}^{-/+}-\eta,\sN}^{\sscr{(1-t)/(t)}}(|\dot{\sigma}|)\wedge C.
\end{align*} 
$\sigma\in \mathcal{G}^{4\epsilon+\diam_{\sX}}\mapsto a^{-/+}(\sigma)$ is continuous function with repsect to $\de_{\infty}$. 
The dependence of $a^{-/+}$ and $b^{-/+}$ on $\VK$, $\eta$, $N$ and $C$ is suppressed in our notation but in step \textbf{6.} we wil also write $a^{-/+}_{\VK}$ if necessary.
\medskip\\
\textbf{3.}
Let $e_t:\mathcal{G}(X_i)\rightarrow X_i$ be the evaluation map. 
We consider $(e_0,e_1)_{\star}\Pi_i=\pi_i$ that is an optimal plan, and $(e_0,e_1):\Gamma_i\rightarrow \supp\pi_i\subset X\times X$. Let $\Pi_{i,y_0,y_1}(d\gamma)$ be the disintegration of $\Pi_i$
with respect to 
$\pi_i$, and let $\pi_{j,i}(y',dy)$ be a disintegration of $\pi_i$ with respect to $\mu_{j,i}$ for $j=0,1$.
We put
\begin{align*}
v_0(y_0):=\int_{X_i}
{\int_{\mathcal{G}(X_i)}\tau_{\VK_{i,\gamma}^-,N'}^{\sscr{(1-t)}}(|\dot{\gamma}|)\Pi'_{i,y_0,y_1}(d\gamma)}
\pi_{j,i}(y_0,dy_1)
\end{align*}
and similar we define
$
v_1(y_1)
$ 
replacing $\tau_{\sk_{i,\gamma}^-,\sN'}^{\sscr{(1-t)}}(|\dot{\gamma}|)$ by $\tau_{\sk_{i,\gamma}^+,\sN'}^{\sscr{(t)}}(|\dot{\gamma}|)$.
\smallskip\\
\begin{align*}
T^{\sscr{(t)}}_{\VK,N'}(\Pi_i|\m_{\sX_{i}})
&=\sum_{j=0,1}\int_{X_i}\left[\int_X{\varrho}_{j}(x_j){Q}'_i(y_j,dx_j)\right]^{1-\frac{1}{N}}v_j(y_j)\m_{\sX_i}(dy_j)\\
&\geq\sum_{j=0,1}\int_{X_i}\int_X{\varrho}_{j}(x_j)^{1-\frac{1}{N}}{Q}'_i(y_j,dx_j)v_j(y_j)\m_{\sX_i}(dy_j)\\
&=\sum_{j=0,1}\underbrace{\int_{X_i}\int_X{\varrho}_{j}(x_j)^{-\frac{1}{N}}\frac{\varrho_{j}(x_j)}{\varrho_{j,i}(y_j)}{Q}'_i(y_j,dx_j)v_j(y_j)\mu_{i}(dy_j)}_{=:(\dagger)_j}
\end{align*}
One has the following identity.
\begin{align*}
(\dagger)_{\scriptscriptstyle{0}}
&=\int_{X_i^2}\int_{X^2}\int\varrho_0(x_0)^{-\frac{1}{N}}{\frac{\varrho_0(x_0)\varrho_1(x_1)}{\varrho_{0,i}(y_0)\varrho_{1,i}(y_1)}{Q}_i'(y_1,dx_1)Q'_i(y_0,dx_0)}b^-(\gamma)\Pi_{i,y_0,y_1}(d\gamma)\pi_i(dy_0,dy_1)\\
&=\int_{X_i^2}\int\underbrace{\int_{X^2}\varrho_0(x_0)^{-\frac{1}{N}}P_{y_0,y_1}(d(x_0,x_1))}_{=:h(y_0,y_1)}b^-(\gamma)\Pi_{i,y_0,y_1}(d\gamma)\pi_i(d(y_0,y_1))\\
&=\int_{X_i^2}\int h(e_0(\gamma),e_1(\gamma))b^-(\gamma)\Pi_{i,y_0,y_1}(d\gamma)\pi_i(d(y_0,y_1))=(\#)
\end{align*}
In the last equality we used that $(e_0,e_1)(\gamma)$ is constant and equal to $(y_0,y_1)$ on the support of $\Pi_{i,y_0,y_1}(d\gamma)$.
\begin{align*}
(\#)&=\int_{X_i}\int\int_{X^2}\left[a^-\left((\Xi_{i,\gamma}(x_0,x_1)\right)+\left(b^-(\gamma)-a^-(\Xi_{i,\gamma}(x_0,x_1))\right)\right]\\
&\hspace{3cm}\times P_{\gamma_0,\gamma_1}(d(x_0,x_1))\Pi_{i,y_0,y_1}(d\gamma)\pi_i(d(y_0,y_1))\\
&=\begin{rcases}{\displaystyle \int\int_{X^2}\varrho_0(e_0(\Xi_{i,\gamma}(x_0,x_1))^{-\frac{1}{N}}a^-(\Xi_{i,\gamma}(x_0,x_1))P_{\gamma_0,\gamma_1}(d(x_0,x_1))\Pi_{i}(d\gamma)}\end{rcases} =(**)_{\scriptscriptstyle{0}}\\
&\hspace{39pt}\begin{rcases}+{\displaystyle\int_{X_i^2}\int\int_{X^2}\varrho_0(x_0)^{-\frac{1}{N}}\left(b^-(\gamma)-a^-(\Xi_{i,\gamma}(x_0,x_1))\right)}\\
\hspace{3cm}\times P_{\gamma_0,\gamma_1}(d(x_0,x_1))\Pi_{i,y_0,y_1}(d\gamma)\pi_i(d(y_0,y_1))\end{rcases}=(*)_{\scriptscriptstyle{0}}
\end{align*}
and similar for $(\dagger)_{\sscr{1}}$. 
%
%
%
%
%
%
\smallskip\\
\textbf{4.}
Consider $m=\inf\left\{\eta>0: \hat{q}_i(\left\{\de_{\sX}>\eta\right\})<\eta\right\}$ and a positive $\eta>m$ such that
$\eta<m+{\epsilon/2}$. By Markov's inequality and since $i\geq i_0$ (for instance see the proof Proposition 2.6 (i) in \cite{sos}) one has 
$$m\leq \left(\int\de_{\sX}^2(x,y)d\hat{q}_i(x,y)\right)^{\frac{1}{3}}<(2(\epsilon^3/2^4))^{\frac{1}{3}}=\epsilon/2.$$
Therefore, it follows $\eta<\epsilon$ and 
$
\hat{q}_i\left(\left\{\de_{\sX}>\epsilon\right\}\right)\leq \hat{q}_i\left(\left\{\de_{\sX}>\eta\right\}\right)\leq \eta <\epsilon.
$
Define $\left\{\de_{\sX}\leq\epsilon\right\}=:\hat{X}^{\epsilon}\subset X\times X$.
\smallskip\\
Let us consider $(*)_{\sscr{0}}$.
\begin{align*}
(*)_{\sscr{0}}
&=\int_{X_i^2}\int\int_{X^4}(b^-(\gamma)-a^-(\Xi_{i,\gamma}(x_0,x_1))) \varrho_0(x_0)^{-\frac{1}{N}}\frac{\varrho_0(x_0){\varrho}_{1}(x_1)}{\varrho_{0,i}(\gamma_0)\varrho_{1,i}(\gamma_1)}\\
&\hspace{10pt}\times\bar{Q}_i'(\gamma_1,dx_1)\hat{Q}_i'(x_1,dz_1)\bar{Q}_i'(\gamma_0,dx_0)\hat{Q}(x_0,dz_0)\Pi_{i,y_0,y_1}(d\gamma)\pi_i(d(y_0,y_1))\\
&=\int_{X_i^2}\int\int_{X^2}(b^-(\gamma)-a^-(\Xi_{i,\gamma}(x_0,x_1))) \varrho_0(x_0)^{1-\frac{1}{N}}\\
&\hspace{10pt}\times\bar{Q}_i'(\gamma_0,dx_0)\hat{Q}(x_0,dz_0)\Pi_{i,y_0,y_1}(d\gamma)\pi_i(y_0,dy_1))\m_{\sX_i}(dy_0)\\
&={\textstyle \underbrace{\int_{X_i^2}\int\int_{\hat{X}^{\epsilon}}(b^-(\gamma)-a^-(\Xi_{i,\gamma}(x_0,x_1)))\dots}_{=:(II)} +\underbrace{\int_{X_i^2}\int\int_{(\hat{X}^{\epsilon})^c}(b^-(\gamma)-a^-(\Xi_{i,\gamma}(x_0,x_1)))\dots}_{=:(I)}}
\end{align*}
Since $a^-$ and $\varrho_0$ are bounded, there exists a constant $M:=M(C)>0$ such that
\begin{align*}
(I)&\geq -\int_{X_i^2}\int\int_{(\hat{X}^{\epsilon})^c}M\bar{Q}_i'(y_0,dx_0)\hat{Q}(x_0,dz_0)\Pi_{i,y_0,y_1}(d\gamma)\pi_i(y_0,dy_1))\m_{\sX_i}(dy_0)\\
&= -\int_{X_i^2}\int_{(\hat{X}^{\epsilon})^c}M\bar{Q}_i'(y_0,dx_0)\hat{Q}(x_0,dz_0)\pi_i(y_0,dy_1)\m_{\sX_i}(dy_0)\\
&= -\int_{X_i}\int_{(\hat{X}^{\epsilon})^c}M\hat{Q}(x_0,dz_0)\bar{Q}_i'(y_0,dx_0)\m_{\sX_i}(dy_0)\\
&= -\int_{X_i}\int_{(\hat{X}^{\epsilon})^c}M\hat{Q}(x_0,dz_0)\bar{Q}^i(x_0,dy_0)\m_{\sX}(dx_0)
=-2M\hat{q}_i((\hat{X}^{\epsilon})^c)\geq -2M\epsilon
\end{align*}
Consider $(II)$. Define measures on $X^2$ as follows
\begin{align*}
&\hat{P}_{y_0,y_1}(A\times B)=\\
&\int_{(X^{\epsilon})^4}1_{A\times B}(z_0,z_1)\frac{\varrho_0(x_0){\varrho}_{1}(x_1)}{\varrho_{0,i}(y_0)\varrho_{1,i}(y_1)}\bar{Q}_i'(y_1,dx_1)\hat{Q}_i'(x_1,dz_1)\bar{Q}_i'(y_0,dx_0)\hat{Q}_i'(x_0,dz_0)
\end{align*}
Then
\begin{align*}
(II)=\int\int \int(b^-(\gamma)-a^-(\sigma))\varrho_0(x_0)^{-\frac{1}{N}}(\Xi_{\gamma,i})_{\star}\hat{P}_{\gamma_0,\gamma_1}(d\sigma)\Pi_{i,y_0,y_1}(d\gamma)\pi_i(d(y_0,y_1)).
\end{align*}
By construction of $\Xi_{i,\gamma}$ we have that $\Xi_{i,\gamma}$ maps the support of $\hat{P}_{\gamma_0,\gamma_1}$ to $\mathcal{G}^{4\epsilon}(X)$ 
and $\Xi_{i,\gamma}\left(\supp\hat{P}_{\gamma_0,\gamma_1}\right)$ is $4\epsilon$-close to $\gamma$ in $Z$.
Therefore $b^-(\gamma)-a^-(\cdot)\geq 0$ on the support of $(\Xi_{i,\gamma})_{\star}\hat{P}_{\gamma_0,\gamma_1}(d\sigma)$.
Hence, $(II)\geq 0$.
We obtain
\begin{align}\label{badweather}
T^{\sscr{(t)}}_{\VK,N'}(\Pi_i|\m_{\sX_{i}})\geq&
\int \left[a^-(\sigma)\varrho_0(\sigma_0)^{-\frac{1}{N}}+ a^+(\sigma)\varrho_1(\sigma_1)^{-\frac{1}{N}}\right]\hat{\Pi}_i(d\sigma)-4M\epsilon
\end{align}
\smallskip\\
\textbf{5.}
Since $\mathcal{G}^{4\epsilon+\diam_{\sX}}(X)$ is compact with respect to $\de_{\infty}$, Prohorov's theorem yields that 
there is a subsequence of $\hat{\Pi}_i$ that converges to a dynamical transference plan $\Pi$ that is supported on $\mathcal{G}^{4\epsilon+\diam_{\sX}}(X)$. By 
a straightforward modification of
Lemma \ref{therethere} (replacing $\tau_{\VK^{-/+},N}$ by $a^{-/+}$) it follows that 
\begin{align*}
\mbox{RHS in (\ref{badweather})}\rightarrow \int \left[a^{-}({\sigma})\varrho_0(\sigma_0)^{-\frac{1}{N}}+ a^{+}({\sigma})\varrho_1(\sigma_1)^{-\frac{1}{N}}\right]{\Pi}(d\sigma) - 4M\epsilon.
\end{align*}
We show that $(e_0,e_1)_{\star}\Pi=:\pi$ is optimal and $\Pi$ is actually supported on $\mathcal{G}(X)$. 
The first claim follows by construction of $\hat{\Pi}_i$. We have an explicite 
representation for the coupling $\hat{\pi}_i$ that is the same coupling as contructed by Sturm in \cite{stugeo2} (more precisely, this is $\bar{q}^r$ on page 154). 
It is an almost optimal coupling between $\mu_0$ and $\mu_1$ and the error becomes small if $i$ is large. 
Therefore, since $\hat{\pi}_i\rightarrow \pi$ weakly and since the the Wasserstein distance is l.s.c. with respect to weak convergence, 
$\pi$ is optimal for $\mu_0$ and
$\mu_1$.
\smallskip\\
For the second claim we decompose $\hat{\Pi}_i$ ($i\geq i_0$) with resepect to $X^{\epsilon}$. This can be done similar as in the construction of $\hat{P}$ from above. 
Consider 
\begin{align*}
\int\left((\Xi_{i,\gamma})_{\star}\hat{P}_{\gamma_0,\gamma_1}\right)(d\sigma)\Pi_i(d\gamma)=\tilde{\Pi}_i(\sigma) \ \mbox{ and } \ \Pi_i-\tilde{\Pi}_i=\overline{\Pi}_i
\end{align*}
By construction $\tilde{\Pi}_i$ is supported on $\mathcal{G}^{4\epsilon}(X)$ that is compact. Therefore, we can consider another subsequence of $\Pi_i$ such that $\tilde{\Pi}_i$ converges to a
measure $\tilde{\Pi}$ supported on $\mathcal{G}^{4\epsilon}(X)$. We can conclude that also $\overline{\Pi}_i\rightarrow \Pi-\tilde{\Pi}$ weakly and $\overline{\Pi}_i(\mathcal{G}^{4\epsilon+\diam_{\sX}}(X))\leq 4M{\epsilon}$ for $i\geq i_0$. 
Thus $(\Pi-\tilde{\Pi})(\mathcal{G}^{4\epsilon+\diam_{\sX}}(X))\leq 4M{\epsilon}$. By a diagonal argument we obtain $(\Pi-\tilde{\Pi})(\mathcal{G}^{4\epsilon+\diam_{\sX}}(X))\leq \epsilon$ and $\supp\tilde{\Pi}\subset \mathcal{G}^{4\epsilon}(X)$ 
for any $\epsilon>0$. Hence $\Pi=\tilde{\Pi}$ and it is supported on $\mathcal{G}(X)$.
\medskip\\
Together with the convergence of $Q^i(\mu_{t,i})$ to $\mu_t$ (see the beginning of step \textbf{2.}), the curvature-dimension condition on $X_i$ and lower semi-continuity of $S_N$, we get
\begin{align}\label{rainy}
S_N(\mu_{t}|\m_{\sX})\leq -\int \left[a^{-}(\sigma)\varrho_0(\sigma_0)^{-\frac{1}{N}}+ a^{+}(\sigma)\varrho_1(\sigma_1)^{-\frac{1}{N}}\right]{\Pi}(d\sigma).
\end{align}
Since $\eta$ was arbitrary, application of another compactness argument yields the inequality for $\VK$ instead $\VK-\eta$. 
\smallskip\\
\textbf{6.}
In the last step we want to remove the remaining assumptions, namely continuity of $\VK$ and boundedness of $\varrho_j$ and $a^{-/+}$.
\smallskip\\
We consider general absolutely continuous probability measures $\mu_i=\varrho_id\m_{\sX}\in \mathcal{P}_2(X,\m_{\sX})$ and an arbitrary optimal coupling $\pi$ of them. We define
\begin{align*}
E_r:=\left\{(x_0,x_1)\in X^2:\varrho_i(x_i)\leq r \mbox{ for } i=0,1\right\}
\end{align*}
and for $i=0,1$ 
$$
\mu_i^r=(p_0)_{\star}\left({\pi(E_r)}^{-1}\pi|_{E_r}\right).
$$
Then $\mu_i^r$ has bounded density and we have $W_2(\mu_i,\mu_i^r)<\epsilon$ for $r>0$ sufficiently large. 
If $\VK$ is lower semi-continuous, we take monotone sequence of continuous functions $\VK_n$ that approximates $\VK$ from below. 
Since we can repeat all the previous steps,
for any pair $(r,n)$ we obtain an optimal dynamical coupling $\Pi_{(r,n)}$ and a Wasserstein geodesic $\mu_t^{(r,n)}$ such that (\ref{rainy}) holds with $\VK$ replaced by $\VK_n$.
The right hand side of (\ref{rainy}) is monotone with respect to $r$ and $\VK_n$. Therefore, we obtain 
\begin{align*}
&S_N(\mu_{t}^{r}|\m_{\sX})
\leq-\pi(E_{r})^{\frac{1}{N}}\int \left[a^{-}_{\VK_{n}}(\gamma)(\varrho_0(\gamma_0)\wedge r)^{-\frac{1}{N}}+ a^{+}_{\VK_{n}}(\gamma)(\varrho_1(\gamma_1)\wedge r)^{-\frac{1}{N}}\right]{\Pi}^{r}(d\gamma)
\\
&\leq-\pi(E_{{r}})^{\frac{1}{N}}\int \left[a^{-}_{\VK_{\hat{n}}}(\gamma)(\varrho_0(\gamma_0)\wedge \hat{r})^{-\frac{1}{N}}+ a^{+}_{\VK_{\hat{n}}}(\gamma)(\varrho_1(\gamma_1)(\varrho_1(\gamma_1)\wedge \hat{r})^{-\frac{1}{N}}\right]{\Pi}^{{r}}(d\gamma).
\end{align*}
for $(r,n)\geq (\hat{r},\hat{n})$. 
Compactnes yields converging subsequences $\Pi_{(r_i,n_i)}$ and $\mu_t^{(r_i,n_i)}$ for $i\rightarrow\infty$
and by the definition of weak convergence the limits of $\Pi$ and $\mu_t$
satisfy 
\begin{align*}
S_N(\mu_{t}|\m_{\sX})
\leq-\int \left[a^{-}_{\VK_{\hat{n}}}(\gamma)(\varrho_0(\gamma_0)\wedge \hat{r})^{-\frac{1}{N}}+ a^{+}_{\VK_{\hat{n}}}(\gamma)(\varrho_1(\gamma_1)\wedge \hat{r})^{-\frac{1}{N}}\right]{\Pi}(d\gamma).
\end{align*}
This follows since $a^{-/+}_{\VK_{\hat{n}}}$ is bounded and continuous and the densities $\varrho_i\wedge \hat{r}$
can be approximated by functions $\psi\in C_b(X)$ (compare with the proof of Lemma \ref{therethere}). 
We let $\hat{r},\hat{n}\rightarrow \infty$. Then the theorem of monotone convergence yields the estimate
\begin{align}\label{rainy3}
S_N(\mu_{t}|\m_{\sX})\leq -\int \left[a^{-}_{\VK}(\sigma)\varrho_0(\sigma_0)^{-\frac{1}{N}}+ a^{+}_{\VK}(\sigma)\varrho_1(\sigma_1)^{-\frac{1}{N}}\right]{\Pi}(d\sigma).
\end{align}
Finally, we let $C\nearrow \infty$. Then 
\begin{align*}
a^{-/+}(\gamma)\nearrow \tau_{\sk_{\gamma}^{-/+},\sN}^{\sscr{(t)}}(|\dot{\gamma}|)\in\mathbb{R}\cup\left\{\infty\right\} 
\end{align*}
for any $\gamma\in \mathcal{G}(X)$ and again by the monotone congergence theorem the left hand side in (\ref{rainy3}) converges to 
\begin{align}\label{rainy2}
S_N(\mu_{t}|\m_{\sX})\leq -\int \left[\tau_{\sk_{\gamma}^{-},\sN}^{\sscr{(1-t)}}(|\dot{\gamma}|)\varrho_0(\gamma_0)^{-\frac{1}{N}}+ \tau_{\sk_{\gamma}^{+},\sN}^{\sscr{(t)}}(|\dot{\gamma}|)\varrho_1(\gamma_1)^{-\frac{1}{N}}\right]{\Pi}(d\gamma).
\end{align}
This finishes the proof.
\qed
\end{proofoftheorem}
\begin{corollary}
Let $(M_i,g_{\sM_i})_{i\in\mathbb{N}}$ be a family of compact Riemannian manifolds such that $$\ric_{\sM_i}\geq \VK_i\ \& \ \dim_{\sM_i}\leq N$$ where $\VK_i:M_i\rightarrow \mathbb{R}$ is a family of equi-continuous functions such that
$\VK_i\geq -C$ for some $C>0$. 
There exists subsequence of $(M_i,\de_{\sM_i},\vol_{\sM_i})$ that converges in 
measured Gromov-Hausdorff sense to $(X,\de_{\sX},\m_{\sX})$, and there exists a subsequence of $\VK_i$ such that $\lim \VK_i= \VK$.
Then $X$ satisfies the condition $CD(\VK,N)$.
\end{corollary}
\begin{proof}
Since there is uniform lower bound for the Ricci curvature, Gromov's compactness theorem yields a converging subsequnce. Then, Gromov's Arzela-Ascoli theorem also yields a uniformily converging subsequence of $\VK_i$ with limit $\VK$.
Finally, if we apply the previous stability theorem, we obtain the result. 
\end{proof}
\begin{remark}
One can also prove the stability of the condition $CD(\VK,N)$ with respepct to \textit{pointed measured Gromov-Hausdorff convergence}. 
For instance, one can follow the proof of Theorem 29.24 in \cite{viltot}.
\end{remark}

\section{Non-branching spaces and tensorization property}
\begin{lemma}
Let $(X,\de_{\sX},\m_{\sX})$ be a non-branching metric measure space that satifies $CD(\VK,N)$. Then, for every $x\in\supp\m_{\sX}$ there exists a unique geodesic between $x$ and $\m_{\sX}$-a.e. $y\in X$.
Consequently, there exists a measurable map $\Psi: X^2\rightarrow \mathcal{G}(X)$ such that $\Psi(x,y)$ is the unique geodesic between $x$ and $y$ $\m_{\sX}\otimes \m_{\sX}$-a.e.\ . 
\end{lemma}
\begin{proof}
Since $\VK$ is bounded from below on any ball $B_R(x)$ by Theorem \ref{useful},
one can adapt the proof of Lemma 4.1 in \cite{stugeo2}.
\end{proof}

\begin{proposition}\label{branigan}
Let $\VK:X\rightarrow \mathbb{R}$ be admissible, $N\geq 1$ and $(X,\de_{\sX},\m_{\sX})$ be a metric measure space that is non-branching. Then the following statements are equivalent
\begin{itemize}
\item[(i)] $(X,\de_{\sX},\m_{\sX})$ satisfies the curvature-dimension condition $CD(\VK,N)$. 
\item[(ii)] For each pair $\mu_0,\mu_1\in\mathcal{P}_2(X,\m_{\sX})$ there exists an optimal dynamical transference plan $\Pi$ such that
\begin{align}\label{something}
\varrho_t(\gamma_t)^{-\frac{1}{N}}\geq \tau_{\sk^-_{\gamma},\sN'}^{\sscr{(1-t)}}(|\dot{\gamma}|)\varrho_0(\gamma_0)^{-\frac{1}{N}}+\tau_{\sk^+_{\gamma},\sN}^{\sscr{(t)}}(|\dot{\gamma}|)\varrho_1(\gamma_1)^{-\frac{1}{N}}.
\end{align}
for all $t\in[0,1]$ and $\Pi$-a.e. $\gamma\in\mathcal{G}(X)$. Here $\varrho_t$ is the density of the push-forward of $\Pi$ under the map $\gamma\mapsto \gamma_t$. That is determined by
\begin{align*}
 \int_Xu(y)\varrho_t(y)d\m_{\sX}(y)=\int u(\gamma_t)d\Pi(\gamma).
\end{align*}
for all bounded measurable functions $u:X\rightarrow \mathbb{R}$.
\end{itemize}
\end{proposition}
\begin{proof}
``$\Leftarrow$'': Let $N'>N$ and $\varrho_id\m_{\sX}=\mu_i\in \mathcal{P}_2(X,\m_{\sX})$ for $i=0,1$. H\"older's inequality yields
\begin{align*}
\varrho_t(\gamma_t)^{-\frac{1}{N'}}&\geq \left(\tau_{\sk^-_{\gamma},\sN'}^{\sscr{(1-t)}}(|\dot{\gamma}|)\varrho_0(\gamma_0)^{-\frac{1}{N}}+\tau_{\sk^+_{\gamma},\sN}^{\sscr{(t)}}(|\dot{\gamma}|)\varrho_1(\gamma_1)^{-\frac{1}{N}}\right)^{\frac{N}{N'}}\\
&\geq \tau_{\sk^-_{\gamma},\sN'}^{\sscr{(1-t)}}(|\dot{\gamma}|)^{\frac{N}{N'}}(1-t)^{(1-\frac{N}{N'})}\varrho_0(\gamma_0)^{\frac{-1}{N'}}+\tau_{\sk^+_{\gamma},\sN'}^{\sscr{(t)}}(|\dot{\gamma}|)^{\frac{N}{N'}}t^{(1-\frac{N}{N'})}\varrho_1(\gamma_1)^{\frac{-1}{N'}}
\end{align*}
Additionally, Lemma \ref{gr} yields the estimate $$\tau_{\sk^{-}_{\gamma},\sN}^{\sscr{(1-t)}}(|\dot{\gamma}|)^{\frac{N}{N'}}(1-t)^{1-\frac{N}{N'}}\geq \tau_{\sk^-_{\gamma},\sN'}^{\sscr{(1-t)}}(|\dot{\gamma}|)$$ and similar for the term
involving $\VK^+_{\gamma}$.
Finally, integrating the previous inequality with respect to $\Pi$ yields the condition $CD(\VK,N)$.
\smallskip\\
``$\Rightarrow$'': Consider probability measures $\mu_i=\varrho_id\m_{\sX}$ for $i=0,1$. Let $\Pi$ be an optimal dynamical coupling.
Since for $\m_{\sX}\otimes \m_{\sX}$-a.e. pair $(x,y)$ there exists a unique geodesic $\gamma_{x,y}$, 
there exist an optimal coupling $\pi$ such that $\Pi$ can be written in the form $\delta_{\gamma_{x,y}}d\pi(x,y)$. 
We consider closed balls of increasing radius $R$ for some fixed point $x_0$. $\VK$ is bounded from below by a constant $\underline{\VK}$ on each Ball, and therefore 
one can follow \cite{stugeo2} and prove a local measure contraction property (in the sense of \cite{stugeo2}) that holds in each ball (for instance see \cite{stugeo2}). 
Hence, we can apply the main result of Cavalletti and Huesmann in \cite{cavallettihuesmann}. It tells us that, if a measure contraction property holds locally on a non-branching space, 
each optimal coupling between absolutely continuous probability 
measures is unique and induced 
by a measurable map. Therefore, the curvature-dimension condition for $\mu_0$ and $\mu_1$ becomes
\begin{align*}
&\int_X\int\varrho_t(\gamma_t)^{-\frac{1}{N}}\delta_{\gamma_{x,y}}(d\gamma)d\pi(x,y)\\
&\geq \int_{X^2}\int\left[\tau_{\sk^-_{\gamma},\sN}^{\sscr{(1-t)}}(|\dot{\gamma}|)\varrho_0(\gamma_0)^{-\frac{1}{N}}+\tau_{\sk^+_{\gamma},\sN}^{\sscr{(t)}}(|\dot{\gamma}|)\varrho_1(\gamma_1)^{-\frac{1}{N}}\right]\delta_{\gamma_{x,y}}(d\gamma)d\pi(x,y).
\end{align*}
Now, we can follow exactly the proof of the corresponding result in \cite{stugeo2}.
\end{proof}
\begin{proposition}\label{prop73}
Let $(X,\de_{\sX},\m_{\sX})$ be a non-branching metric measure space that satisfies $CD(\VK,N)$, let $\VK':X\rightarrow \mathbb{R}$ be lower semi-continuous and let $V:X\rightarrow [0,\infty)$ be strongly $\VK'V$-convex in the sense of Definition \ref{ahr}.
Then $(X,\de_{\sX},V^{N'}\m_{\sX})$ satisfies the condition $CD(\VK+\VK',N+N')$.
\end{proposition}
\begin{proof}
The proof is a straighfoward calculation using the characterization of $CD(\VK,N)$ for non-branching spaces, Corollary \ref{somethingmore} and H\"older's inequality.
\end{proof}

\begin{theorem}\label{tensorization}
Let $(X_i,\de_{\sX_i},\m_{\sX_i})$ be non-branching metric measure spaces for $i=1,\dots,k$ statisfying the condition $CD(\VK_i,N_i)$ for admissible functions $\VK_i:X_i\rightarrow \mathbb{R}$ and $N_i\geq 1$. 
Then the metric measure space
\begin{align*}
\left(\Pi_{i=1}^kX_i,\ \sqrt{\sum\nolimits_{i=1}^k\de_{\sX_i}^2},\ \bigotimes\nolimits_{i=1}^k\m_{\sX_i}\right)=(Y,\de_{\sY},\m_{\sY})
\end{align*}
satisfies the condition 
\begin{align*}
CD\left(\min_{i=1,\dots,k}\VK_i,\max_{i=1,\dots,k}N\right)
\end{align*}
where $(\min_{i=1,\dots, k}\VK_i)(x_1,\dots,x_k)=\min\left\{\VK_i(x_i):i=1,\dots,k\right\}$.
\end{theorem}
\begin{proof}
It is enough to consider $k=2$ and measures of $\mu_0$ and $\mu_1$ in $\mathcal{P}_2(Y,\m_{\sY})$ of the form 
$
\mu_0=\mu_0^{\sscr{(1)}}\otimes \mu_0^{\sscr{(2)}} \mbox{ and } \mu_1=\mu_1^{\sscr{(1)}}\otimes \mu_1^{\sscr{(2)}}.
$
Then general case follows in the same way as in \cite{bast} for instance. Consider dynamical optimal couplings $\Pi^{\sscr{(i)}}$ for $\mu_0^{\sscr{(i)}}$ and $\mu_0^{\sscr{(i)}}$ such that (\ref{something}) 
holds according to our curvature assmuption. Let $(e_0,e_1)_{\star}\Pi^{\sscr{(i)}}=\pi^{\sscr{(i)}}$.
The pushforward of $\pi^{\sscr{(1)}}\otimes \pi^{\sscr{(2)}}$ with respect to 
\begin{align*}
(x_0^{\sscr{(1)}}, x_1^{\sscr{(1)}},x_0^{\sscr{(2)}}, x_1^{\sscr{(2)}})\mapsto (x_0^{\sscr{(1)}},x_0^{\sscr{(2)}}, x_1^{\sscr{(1)}}, x_1^{\sscr{(2)}})
\end{align*}
becomes an optimal coupling $\pi$ between $\mu_0$ and $\mu_1$. There is also a measurable map $(\gamma^{\sscr{(1)}},\gamma^{\sscr{(2)}})\in \mathcal{G}(X_1)\times \mathcal{G}(X_2)\mapsto (\gamma^{\sscr{(1)}},\gamma^{\sscr{(2)}})\in \mathcal{G}(Z)$.
Therefore, we can consider the pushforward $\Pi$ of $\Pi^{\sscr{(1)}}\times \Pi^{\sscr{(2)}}$ with respect to this map. Since $(e_0,e_1)_{\star}\Pi=\pi$, $\Pi$ is an optimal dynamical plan for $\mu_0$ and $\mu_1$.
\smallskip\\
\textit{Claim: For geodesics $\gamma^{\sscr{(1)}}\in\mathcal{G}(X_1)$ and $\gamma^{\sscr{(2)}}\in\mathcal{G}(X_2)$ consider $\gamma=(\gamma^{\sscr{(1)}},\gamma^{\sscr{(2)}})\in\mathcal{G}(Y)$, then we have}
\begin{align*}
\tau^{\sscr{(t)}}_{\sk_{1,\gamma},\sN_1}(|\dot{\gamma}^{\sscr{(1)}}|)^{\sN_1}\cdot \tau^{\sscr{(t)}}_{\sk_{2,\gamma},\sN_2}(|\dot{\gamma}^{\sscr{(2)}}|)^{\sN_2}\geq \tau^{\sscr{(t)}}_{\VK_{\gamma},\sN_1+\sN_2}\left(|\dot{\gamma}|\right)^{\sN_1+\sN_2}
\end{align*}
The claim follows immediately from Corollary \ref{somethingmore} combined with the observations that $\tau^{\sscr{(t)}}_{\sk_{\gamma},\sN}(|\dot{\gamma}|)=\tau^{\sscr{(t)}}_{\sk_{\gamma}|\dot{\gamma}|^2,\sN}(1)$, that $|\dot{\gamma}|^2=|\dot{\gamma}^{\sscr{(1)}}|^2+|\dot{\gamma}^{\sscr{(2)}}|^2$, and that 
\begin{align*}
\VK_i\circ \bar{\gamma}^{\sscr{(i)}}(t|\dot{\gamma}^{\sscr{(i)}}|)
=\VK_i\circ\gamma^{\sscr{(i)}}(t)\geq \min_{i=1,2}\left\{\VK_i\circ\gamma(t)\right\}
=\left(\min\nolimits_{i=1,2}\VK_i\circ\bar{\gamma}\right)(t|\dot{\gamma}|).
\end{align*}
for $i=1,2$.
The rest of the proof works exactly like the proof of the corresponding result in \cite{dengsturm}.
%
%
%
\end{proof}
\section{Globalization of the reduced curvature-dimension condition}
\begin{definition}
If we replace in Definition \ref{bigg}
$$\tau_{\VK^{\sscr{-/+}}_{\gamma},N'}^{\sscr{(1-t)/(t)}}(|\dot{\gamma}|)\mbox{ by }\sigma_{\VK^{\sscr{-/+}}_{\gamma},N'}^{\sscr{(1-t)/(t)}}(|\dot{\gamma}|).$$
we say $(X,\de_{\sX},\m_{\sX})$ satisfies the \textit{reduced curvature-dimension condition} $CD^*(\VK,N)$. 
Obviously, we always have that $CD(\VK,N)$ implies $CD^*(\VK,N)$.  
\smallskip\\
We say that $(X,\de_{\sX},\m_{\sX})$ satisfies the the curvature-dimension condition \textit{locally} - denoted by $CD_{loc}(\VK,N)$ - if for 
any point $x$ there exists a neighborhood $U_x$ such that for each pair $\mu_0,\mu_1\in \mathcal{P}_2(X,\m_{\sX})$ with bounded support in $U_x$, one can find
a geodesic $\mu_t\in \mathcal{P}_2(X,\m_{\sX})$ and an optimal dynamical coupling $\Pi\in \mathcal{P}(\mathcal{G}(X))$ such that (\ref{curvaturedimension}) holds. 
Similar, we define $CD^*_{loc}(\VK,N)$.
\end{definition}
\begin{remark}
All the previous results of this article also hold for the condition $CD^*(\VK,N)$ though constants and estimates are in general not sharp. 
\end{remark}
\begin{theorem}\label{globalization}
Let $(X,\de_{\sX},\m_{\sX})$ be a non-branching and geodesic metric measure space with $\supp\m_{\sX}=X$. Let $\VK:X\rightarrow \mathbb{R}$ be admissible.
Then the curvature dimension condition $CD^*(\VK,N)$ holds if and only if it holds locally.
\end{theorem}
\begin{proof}
We only have to show the implication $CD^*_{loc}(\VK,N)$ implies $CD^*(\VK,N)$. Let us assume the curvature-dimension condition holds locally. Therefore, a Bishop-Gromov volume growth result holds locally, and it implies the space is locally compact.
Then the metric Hopf-Rinow theorem implies that $X$ is proper.
Hence, we can assume that $X$ is compact. Otherwise, we choose an exhaustion of $X$ with compact balls $\overline{B_R(o)}$ such that the optimal transport between measures supported in $B_R(o)$ does not leave $\overline{B_{2R}(o)}$. 
For instance, compare with the proof of Theorem 5.1 in \cite{bast}. Similar as in the proof of Proposition \ref{branigan} one can also see that a measure contraction property holds locally. Then, the result of \cite{cavallettihuesmann} implies
uniqueness of $L^2$-Wasserstein geodesics.
\smallskip\\
By compactness of $X$ there is $\lambda \in (0,\diam_{\sX})$, finitely many disjoints sets
$L_1,\dots, L_k$ that cover $X$ and have non-zero measure, and finitely many open sets $M_1,\dots, M_k$ such that $B_{\lambda}(L_i)\subset M_i$ for $i\in \left\{1,\dots k\right\}$ and such that
(\ref{curvaturedimension}) holds in $M_i$ for each $i$ (for instance, see the proof of Theorem 5.1 in \cite{bast}). 
\smallskip\\
Let $\mu_0,\mu_1\in \mathcal{P}_2(X,\m_{\sX})$ be arbitrary and let $\mu_t$ be the $L^2$-Wasserstein geodesic between $\mu_0$ and $\mu_1$. Consider $\mu_{\bar{t}}$ and $\mu_{\bar{s}}$ such that 
$\bar{s}-\bar{t}\leq \lambda/\diam_{\sX}$. We define $\nu_{\tau}=\mu_{(1-\tau)\bar{t}+\tau \bar{s}}$ is a geodesic between $\mu_{\bar{t}}$ and $\mu_{\bar{s}}$, and any transport geodesic has length less than
$\lambda$. $\Pi$ denotes the optimal dynamical transference plan that corresponds to $\nu_t$.
We decompose $\nu_0$ with respect to $(L_i)_{i=1,\dots,k}$ as follows
\begin{align*}
\nu_0=\sum_{i=1}^{k}\frac{1}{\nu_0(L_i)}\nu_0|_{L_i}=\sum_{i=1}^{k}\nu_0^i.
\end{align*}
Define $\mathcal{L}_i=\left\{\gamma\in \mathcal{G}(X):\gamma(0)\in L_i\right\}$ with $\nu_0(L_i)=\Pi(\mathcal{L}_i)$. The restriction property of optimal transport yields that $\Pi^i=\Pi(\mathcal{L}_i)^{-1}\Pi|_{\mathcal{L}_i}$ 
are optimal dynamical couplings between $\nu_0^i$ and $\nu_1^i=(e_1)_*\Pi^i$ and $\Pi^i$ induces a geodesic $\nu_{\tau}^i$ between $\nu^i_0$ and $\nu^i_1=(e_1)_*\Pi^i$.
By construction $\nu_1^i$ is supported in $M_i$. Hence, the condition $CD(\VK,N)$ implies
\begin{align*}
{\varrho}_t^i(\gamma(t))^{-\frac{1}{N}}\geq \sigma_{\sk^-_{\gamma},\sN}^{\sscr{(1-t)}}(|\dot{\gamma}|){\varrho}_0^i(\gamma(0))^{-\frac{1}{N}}+\sigma_{\sk^+_{\gamma},\sN}^{\sscr{(t)}}(|\dot{\gamma}|){\varrho}_1^i(\gamma(1))^{-\frac{1}{N}}
\end{align*}
for $\Pi^i$-a.e. $\gamma\in \mathcal{G}(X)$ where $\varrho_t^id\m_{\sX}=d\nu_t^i$. In particular, $\nu_t$ is abslutely continuous with density $\varrho_t=\sum_{i=1}^k\varrho_t^i$.
\smallskip\\
The measures $\nu_0^i$ are disjoint. Therefore, the measures $\nu_t^i$ for $i=1,\dots,k$ are disjoint for any $t\in [0,1)$ (see for instance Lemma 2.6 in \cite{bast}). 
Since any optimal transport between absolutely continuous probability measures is induced by an optimal map, we can conclude that also $\nu_1^i$ are disjoint.  
Therefore, for any $t\in [0,1]$
\begin{align}
\varrho_t(x)^{-\frac{1}{N}}= \sum_{i=1}^k \frac{1}{\Pi(\mathcal{L}_i)}\varrho_t^i(x)^{-\frac{1}{N}}
\end{align}
where $\varrho_t^i d\m_{\sX}=d\nu_t^i$. Hence
\begin{align*}
{\varrho}_t(\gamma(t))^{-\frac{1}{N}}\geq \sigma_{\sk^-_{\gamma},\sN}^{\sscr{(1-t)}}(|\dot{\gamma}|){\varrho}_0(\gamma(0))^{-\frac{1}{N}}+\sigma_{\sk^+_{\gamma},\sN}^{\sscr{(t)}}(|\dot{\gamma}|){\varrho}_1(\gamma(1))^{-\frac{1}{N}}
\ \ \mbox{for $\Pi$-a.e. $\gamma\in \mathcal{G}(X)$.}\end{align*}
In particular, the previous argument holds for each $\bar{s},\bar{t}\in [0,1]\cap \mathbb{Q}$. Thus, if $\mu_t$ is the unique geodesic between $\mu_0,\mu_1$ and $\Pi$ is the corresponding optimal dynamical plan, we showed that
\begin{align*}
{\rho}_{\tau(t)}(\gamma(\tau(t)))^{-\frac{1}{N}}\geq \sigma_{\sk^-_{\gamma},\sN}^{\sscr{(1-\tau(t))}}((s-t)|\dot{\gamma}|){\rho}_t(\gamma(t))^{-\frac{1}{N}}+\sigma_{\sk^+_{\gamma},\sN}^{\sscr{(\tau(t))}}((s-t)|\dot{\gamma}|){\rho}_s(\gamma(s))^{-\frac{1}{N}}
\end{align*}
for $\Pi$-a.e. geodesic $\gamma$ and each $\bar{t},\bar{s}\in[0,1]\cap\mathbb{Q}$ where $\tau(t)=(1-t)\bar{t}+t\bar{s}$. If we pick such a geodesic $\gamma$, the inequality holds also globally along $\gamma$ for $\rho_t$ by Corollary \ref{central2}.
Then the result follows.
\end{proof}